\documentclass[compress, preprint,9pt]{elsarticle}

\usepackage[abs]{overpic}
\usepackage{graphicx}
\usepackage{subfigure}
\usepackage{stmaryrd}
\usepackage{amsfonts, color, amssymb}
\usepackage{amsmath}
\usepackage{amsthm}
\usepackage{epsfig}
\usepackage{psfrag}
\usepackage{bm}
\usepackage{paralist}
\usepackage{algorithm}
\usepackage{algorithmicx}
\usepackage{algpseudocode}
\usepackage{appendix}
\usepackage{caption}
\usepackage{hyperref}
\usepackage{color}
\usepackage{mathtools}

\usepackage[misc]{ifsym}

\usepackage[a4paper, total={6in,8in}]{geometry}

\allowdisplaybreaks[4]
\newtheorem{theorem}{Theorem}[section] %
\newtheorem{remark}{Remark}[section]

\newtheorem{example}{Example}[section]

\newtheorem{lemma}{Lemma}[section]
\newcommand{\bx}{{\bf x}}

\newcommand{\bw}{{\bf w}}

\newcommand{\bv}{{\mathbf v}}

\def\bg{{\bf g}}

\def\bn{{\bf n}}
\def\3bar{{|\hspace{-.02in}|\hspace{-.02in}|}}

\allowdisplaybreaks[4]
\def\jump#1{{[\![#1]\!]}}
\newcommand{\trb}[1]{|\!|\!|#1|\!|\!|}

\allowdisplaybreaks 

\numberwithin{equation}{section}

\begin{document}

\begin{frontmatter}

\title{The Immersed Discontinuous Galerkin Method for Elliptic Interface Problems}

\author[mymainaddress]{Lin Yang}
\ead{yanglin25@amss.ac.cn}
\author[mysecondaddress]{Qilong Zhai}
\ead{zhaiql@jlu.edu.cn}

\address[mymainaddress]{Academy of Mathematics and Systems Science, Chinese Academy of Sciences, Beijing 100190, P. R. China} 
\address[mysecondaddress]{School of Mathematics, Jilin University, Changchun 130012, Jilin, P. R. China}

\begin{abstract}
This paper is devoted to construction and convergence analysis of the linear explicit immersed finite element (IFE) function. For the interface elements, the proposed IFE functions precisely satisfy the interface conditions on the actual interface. The IFE functions are constructed in an explicit form and can be obtained directly without solving any auxiliary problems or local linear systems. Although the constructed IFE functions are non-polynomial, we establish rigorous theoretical analysis showing that they achieve optimal approximation properties and satisfy the essential trace inequalities. And the constants in the analysis are independent of how the interface cuts through the elements. Based on these IFE functions, an immersed discontinuous Galerkin numerical scheme is developed. Several numerical experiments are implemented to confirm that both the IFE functions and the numerical method achieve optimal convergence rates in the $H^1$ and $L^2$ norms. Furthermore, the numerical results indicate that the condition numbers of the stiffness matrices are robust with respect to the interface location.
\end{abstract}


\end{frontmatter}


\section{Introduction}
Consider a bounded domain $\Omega \subset \mathbb{R}^2$, which is partitioned by the {smooth} interface~$\Gamma$ into two subdomains~$\Omega_1$ and~$\Omega_2$ with $\Gamma = \partial\Omega_1 \cap \partial\Omega_2$. In this work, we study the following second-order elliptic interface problems:
\begin{align}
	-\nabla \cdot (\beta \nabla u) &= f, \quad \text{in} \, \Omega_1 \cup \Omega_2, \label{Elliptic_problem}\\
	u &= g ,\quad \text{on} \, \partial \Omega, \label{Dirichlet_Boundary}
\end{align}

where $f \in L^2(\Omega)$, $g \in H^{\tfrac{1}{2}}(\partial \Omega)$, and $\beta$ is a piecewise constant function
\begin{eqnarray*}
	\beta=\left\{
		\begin{array}{rl}
			\beta_1, &(x,y) \in \Omega_1, \\
			\beta_2, & (x,y) \in \Omega_2.
		\end{array}\right.
\end{eqnarray*}
{Without loss of generality, we suppose $\beta_2 \geqslant \beta_1 > 0$.} Set $u_1=u|_{\Omega_{1}}$ and $u_2=u|_{\Omega_{2}}$.
The {homogeneous} interface conditions on $\Gamma$ are as follows:
\begin{align}
	\jump{u}_{\Gamma}&:= u_1 - u_2 =0, \qquad\qquad\qquad\,\,\,\quad \text{on} \, \Gamma, \label{Interface_condition_1}\\
	\jump{\beta \nabla u \cdot \bn}_{\Gamma}&:= \beta_1 \nabla u_1 \cdot \bn - \beta_2 \nabla u_2 \cdot \bn = 0, \quad \text{on}\, \Gamma,\label{Interface_condition_2}
\end{align}
where $\bn$ is the unit outward normal vector on the interface $\Gamma$ pointing from $\Omega_1$ into $\Omega_2$. The interface problems under consideration appear in many applications involving multiply materials and interfaces, such as electrostatic levitation of dust particles~\cite{Interface_application_3}, multi-phase fluid simulation ~\cite{Interface_application_4,Interface_application_5}. 

Many numerical methods were proposed for solving interface problems. According to the relationship between the mesh and the location of interface, these methods can generally be classified into two categories: \textit{fitted finite element methods (FEMs)} and \textit{unfitted FEMs}. In \textit{fitted FEMs}, the interface is consistent with the boundary of the element, so standard FEMs can be directly applied to solve the above interface problems \cite{FEM_1,FEM_2,FEM_3}. However, for interface problems with complex or moving interfaces, the fitted mesh often needs to be regenerated or updated to accurately capture the evolving geometry, which results in considerable computational cost. In contrast, the unfitted mesh allows the interface to cut through the interior of elements and are therefore more suitable for such problems. Existing \textit{unfitted FEMs} mainly include the immersed interface method \cite{IIM2,li2006immersed} and the CutFEM \cite{CutFEM1,CutFEM2}, the partition of unity method \cite{PIU1,PIU2}, the extended finite element method \cite{ExtendedFEM1,ExtendedFEM2}, generalized finite element method \cite{GFEM1,GFEM2,GFEM3,GFEM4} and immersed finite element  (IFE) method \cite{ IFEM_Loworder9,IFEM_Loworder2,IFEM_Loworder5,IFEM1998,IFEM_Loworder4,Partially_penalized_IFEM_2015,IFEM_Loworder6,IFEM_Loworder7}.

In this paper, we mainly focus on the IFE method proposed in \cite{IFEM1998}. This method employs the standard finite element function on the non-interface elements and constructs the IFE functions that satisfy the interface conditions on the interface elements. The number of global degrees of freedom is the same as in the standard finite element method, making this method more attractive and advantageous for solving interface problems on unfitted meshes. The key point of the IFE method is how to construct IFE functions and prove that the constructed IFE functions achieve optimal approximation. Based on the homogeneous interface conditions and the values at the element vertices, the authors in \cite{IFEM2003} derived and solved a local linear system to construct both conforming and nonconforming linear IFE spaces on the interface elements. Instead of employing nodal values to construct the IFE functions, \cite{IFEM2010} developed the nonconforming IFE functions with the average values on the boundary of the element. In \cite{IFE2019}, the authors proposed a unified framework of linear, bilinear, and $Q_1$ type IFE functions and proved the optimal interpolation estimates through the multipoint Taylor expansion method. For high-order IFE methods, in order to avoid the geometric errors caused by approximating curved interfaces with straight segments, many researchers constructed IFE functions based on the actual curved interfaces. In \cite{IFEHigh_2019}, the authors presented IFE functions that weakly satisfy the interface conditions by solving local Cauchy problems on the interface elements. They also proved the optimal approximation properties of these functions. More recently, based on the geometric properties of the interface, \cite{IFEHigh_2024} proposed IFE functions that exactly satisfy the interface conditions, and the corresponding optimal approximation analysis was established. Furthermore, \cite{IFE_2025} further proved the optimal error analysis for the DG method.

For the aforementioned IFE functions, a local problem usually needs to be solved on each interface element to determine the IFE functions. This increases the computational cost. Moreover, when small-cut elements occur, the condition numbers of the local coefficient matrices may become large. As a result, the numerical accuracy may deteriorate, or additional preconditioning techniques may be required to improve the stability of the local solvers. To address this issue, \cite{IFE_2022} presented unified linear explicit IFE functions for both two- and three-dimensional problems based on the standard Crouzeix-Raviart (CR) element. Subsequently, \cite{IFE_2026} proposed linear explicit IFE functions for anisotropic elliptic interface problems with nonhomogeneous interface conditions.

 These existing explicit IFE functions were constructed based on approximate interfaces, which may introduce geometric errors when applied to high-order methods. Thus, in this work, we present linear explicit IFE functions based on the actual interface rather than on an approximate interface. These IFE functions exactly satisfy the interface conditions. Unlike many existing IFE approaches, the proposed IFE functions are constructed in an explicit form and can be obtained directly without solving any auxiliary local problems or local linear systems on interface elements. This avoids additional computational cost associated with local solvers. We establish a rigorous theoretical analysis to show that the proposed IFE functions achieve optimal approximation properties and satisfy the essential trace inequalities. In particular, the constants appearing in the analysis are independent of the interface location. Based on these IFE functions, we further develop an immersed symmetric interior penalty discontinuous Galerkin (SIPDG) method \cite{SIPDG_1982,riviere2008discontinuous}. Following the standard arguments, optimal convergence orders are obtained in both the $H^1$ and $L^2$ norms. Numerical experiments also demonstrate that the condition numbers of the stiffness matrices are uninfluenced by small-cut elements.

This paper is organized as follows. In Section 2, we give the geometric properties of the interface and the inequalities used in the projection estimates. In Section 3, the immersed discontinuous Galerkin method is proposed for solving second-order elliptic interface problems. Section 4 proves the optimal projection estimates of IFE functions. In Sections 5 - 6, the important inequality for the IFE function is established, and optimal error estimates are proved in both the $H^1$ and $L^2$ norms. In Section 7, we implement several numerical examples to investigate the performance of the proposed IFE functions and numerical scheme.

\section{Geometric Properties of The Interface}
In this section, we discuss several geometric properties of the interface that play a fundamental role in the analysis of the IFE function. This section closely follows the presentation in \cite{IFEHigh_2024,IFE2019}. For the reader's convenience and ease of reference in subsequent sections, we reproduce the relevant material here. 

Assume that the interface $\Gamma$ is implicitly defined as the zero set of a non-degenerate level-set function $\ell$:
$$
\Gamma=\{X=(x,y)\in \mathbb{R}^2: \ell(x,y)=0\},
$$
where $\ell(x,y)$ is $C^1$ smooth function in a neighborhood of $\Gamma$, such that 
\begin{align*}
	\ell(x,y) < 0 \quad \text{in} \, \Omega_1, \qquad 	\ell(x,y) > 0 \quad \text{in} \, \Omega_2, \qquad
	|\nabla \ell|\geqslant c_0 >0 \quad \text{in}\, U_{\delta_0}\subset \Omega.
\end{align*}
Here $U_{\delta_0}\subset \Omega$ is a tubular neighborhood of $\Gamma$ of width $\delta_0$: $U_{\delta_0}=\{X=(x,y)\in \mathbb{R}^2: \text{dist}(X,\Gamma)< \delta_0\}$, with $\delta_0>0$ a sufficiently small constant. A special choice for $\ell(x,y)$ is the signed distance function to $\Gamma$. For the subsequent estimate to hold, we assume that the level-set function $\ell$ has the smoothness property {$\ell \in C^3(U_{\delta_0})$}.

Let $T$ be the interface element. For a point $X=(x,y)$ on $\Gamma$, let $\bn(X)=(n_x(X),n_y(X))$ be the unit normal vector of $\Gamma$ at $X$.

\begin{lemma}\cite{IFE2019}
	For the function $u(x,y)$ satisfying the interface conditions \eqref{Interface_condition_1} - \eqref{Interface_condition_2} and the arbitrary point $X$ on $\Gamma \cup T$, we have
	\begin{eqnarray}
		\nabla u_{2}(X)=N_{1}(X)\nabla u_{1}(X), \,N_{1}(X)=\begin{pmatrix}
			n_y^2(X)+\rho n_x^2(X) & (\rho-1)n_x(X)n_y(X)\\
			(\rho-1)n_x(X)n_y(X) & 	n_x^2(X)+\rho n_y^2(X)
		\end{pmatrix},\label{nablau2nablau1}\\
		\nabla u_{1}(X)=N_{2}(X)\nabla u_{2}(X), \,N_{2}(X)=\begin{pmatrix}
			n_y^2(X)+\widetilde{\rho} n_x^2(X) & (\widetilde{\rho}-1)n_x(X)n_y(X)\\
			(\widetilde{\rho}-1)n_x(X)n_y(X) & 	n_x^2(X)+\widetilde{\rho} n_y^2(X)
		\end{pmatrix},\label{nablau1nablau2}
	\end{eqnarray}
where $\rho=\dfrac{\beta_1}{\beta_2}$ and $\widetilde{\rho}=\dfrac{\beta_2}{\beta_1}$.
\end{lemma}

\begin{lemma}\cite{IFE2019}
	For the mesh $\mathcal{T}_h$ with $h$ sufficiently small and the arbitrary point $X$ on $\Gamma \cup T$, the follows estimates hold true
	\begin{eqnarray}\label{Estimate_N}
		\|N_1(X)\| \leqslant C, \,\, \|N_2(X)\| \leqslant \frac{C}{\rho},
	\end{eqnarray}
	where the constant $C$ is independent of interface location, the viscosity coefficients $\beta_1$ and $\beta_2$, and the maximum curvature $\kappa$ of the curved interface $\Gamma$.
\end{lemma}
 
Next, we assume that the parametrization of the interface $\Gamma$ is as follows:
$$
\bg(\xi)=(g_1(\xi),g_2(\xi)),\, [\xi_s, \xi_e] \rightarrow \Gamma,
$$ 
and is regular in the sense that $\bg'(\xi) \neq 0$ for all $\xi \in [\xi_s, \xi_e]$. The unit tangent vector ${\bm\tau}(\xi)$ at a point $\bg(\xi) \in \Gamma$  is given by 
$$
{\bm\tau}(\xi)=\begin{pmatrix}
	\tau_1(\xi)\\
	\tau_2(\xi)
\end{pmatrix}=\frac{1}{\|\bg'(\xi)\|} \bg'(\xi),
$$
the unit normal vector is denoted by
$$
\bn(\xi)=\begin{pmatrix}
	n_1(\xi)\\
	n_2(\xi)
\end{pmatrix}=Q{\bm\tau}(\xi),\, Q=\begin{pmatrix}
		0 & 1\\
	-1 & 0
\end{pmatrix},
$$
and the signed curvature $\kappa(\xi)$ is defined as
$$
\kappa(\xi)=\frac{1}{\|\bg'(\xi)\|^3}(\bg'(\xi) \wedge \bg''(\xi)),\, \text{where}\, \bv \wedge \bw= \bv^{T} Q \bw.
$$

In the neighborhood of the interface $\Gamma$, there exist a family of curves that are locally parallel to $\Gamma$. Based on the Frenet frame, a curve parallel to $\Gamma$ with an offset distance $\eta$ from $\Gamma$ has the following parametric form \cite{IFEHigh_2024}:
\begin{eqnarray*}
	\bx(\eta, \xi)=\begin{pmatrix}
		x(\eta, \xi)\\
		y(\eta, \xi)
	\end{pmatrix}
=P_{\Gamma}(\eta, \xi) = \bg(\xi)+\eta \bn(\xi),\, \xi \in [\xi_s, \xi_e].
\end{eqnarray*} 
Here $\eta$ measures the signed distance from the interface along the normal direction $\bn(\xi)$. 

We assume that the interface~$\Gamma$ admits a tubular neighborhood 
\[
N_{\Gamma}(\varepsilon)
= P_{\Gamma}([-\varepsilon, \varepsilon] \times [\xi_s, \xi_e]),
\]
with a half-band width $\varepsilon > 0$, such that 
\begin{itemize}
	\item[(i)] for every point $(x(\eta,\xi),y(\eta,\xi)) \in N_\Gamma(\varepsilon)$, 
	the Jacobian $\widehat{D}P_\Gamma(\eta,\xi)$ of the Frenet transformation at $(\eta,\xi)$ is non-singular;
	\item[(ii)] the lines perpendicular to $\Gamma$ passing through any two distinct points on $\Gamma$ 
	do not intersect within this tubular neighborhood $N_\Gamma(\varepsilon)$;
	\item[(iii)] the mapping 
	$P_\Gamma : [-\varepsilon, \varepsilon] \times [\xi_s, \xi_e] \to N_\Gamma(\varepsilon)$ 
	is one-to-one and onto.
\end{itemize}
Since $P_\Gamma $ 
is a one-to-one and onto mapping, it has the inverse mapping 
$R_{\Gamma}: N_{\Gamma}(\varepsilon) \rightarrow [-\varepsilon,\varepsilon] \times [\xi_s, \xi_e]$ such that
$$
\begin{pmatrix}
	\eta(x,y)\\\xi(x,y)
\end{pmatrix}=R_{\Gamma}(x,y)=P^{-1}_{\Gamma}(x,y)\in [-\varepsilon,\varepsilon] \times [\xi_s, \xi_e], \, \, (x,y) \in N_{\Gamma}(\varepsilon).
$$

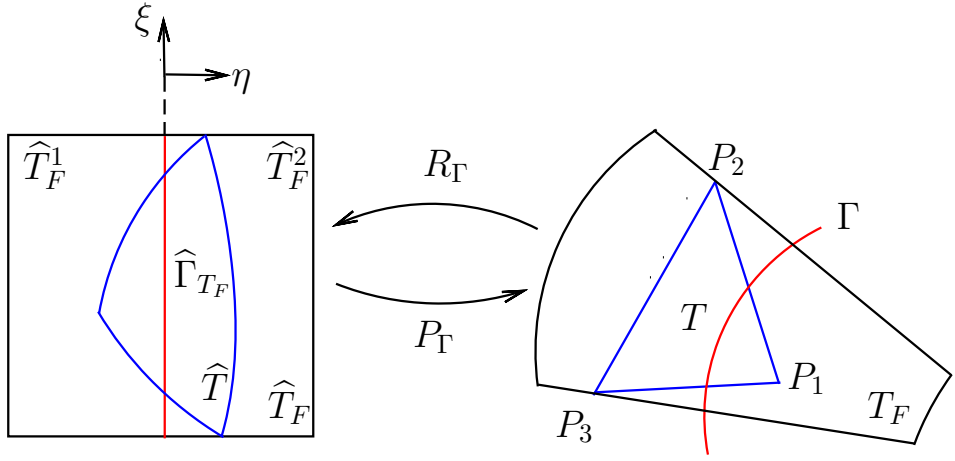
\begin{figure}
	\centering
	\setlength{\unitlength}{1bp}%
	\begin{picture}(365.01, 181.80)(0,0)
		\put(0,0){\includegraphics{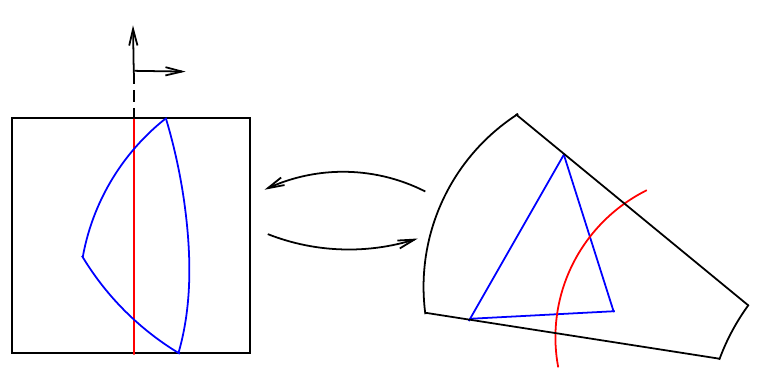}}
		\put(52.97,165.02){\fontsize{14.23}{17.07}\selectfont $\xi$}
		\put(89.39,143.58){\fontsize{14.23}{17.07}\selectfont $\eta$}
		\put(77.93,25.96){\fontsize{14.23}{17.07}\selectfont $\widehat{T}$}
		\put(11.26,108.91){\fontsize{14.23}{17.07}\selectfont $\widehat{T}^1_F$}
		\put(102.38,108.91){\fontsize{14.23}{17.07}\selectfont $\widehat{T}^2_F$}
		\put(102.95,19.40){\fontsize{14.23}{17.07}\selectfont $\widehat{T}_F$}
		\put(67.70,68.49){\fontsize{14.23}{17.07}\selectfont $\widehat{\Gamma}_{T_F}$}
		\put(257.64,52.37){\fontsize{14.23}{17.07}\selectfont $T$}
		\put(297.64,30.37){\fontsize{14.23}{17.07}\selectfont $P_1$}		\put(267.64,112.37){\fontsize{14.23}{17.07}\selectfont $P_2$}		\put(210.64,12.37){\fontsize{14.23}{17.07}\selectfont $P_3$}		
		\put(316.28,90.23){\fontsize{14.23}{17.07}\selectfont $\Gamma$}
		\put(327.79,19.33){\fontsize{14.23}{17.07}\selectfont $T_F$}
		\put(161.08,110.04){\fontsize{14.23}{17.07}\selectfont $R_{\Gamma}$}
		\put(157.41,45.99){\fontsize{14.23}{17.07}\selectfont $P_{\Gamma}$}
	\end{picture}%
	\caption{The transformation between the interface element  $T$  and $\widehat{T}$. }
	\label{inteface_element_reference_transform_IEI}
\end{figure}

For the interface element $T=\triangle P_1P_2P_3$, we construct a fictitious element $T_F$ to cover $T$. For each vertex $P_i$, $1 \leqslant i \leqslant 3$, there exists a unique $\xi_i \in [\xi_s, \xi_e]$ such that 
	\begin{align*}
		\|P_i - \bg(\xi_i)\|={\rm{dist}}(\Gamma, P_i).
	\end{align*}
	Therefore, we get two parameters $a_T$, $b_T \in [\xi_s, \xi_e]$ such that
	\begin{align*}
		a_T=\min(\xi_1,\xi_2,\xi_3), \quad b_T=\max(\xi_1,\xi_2,\xi_3).
	\end{align*}
	Then, we construct a rectangle $\widehat{T}_F=[-h,h] \times [a_T,b_T]$ and {curved trapezoid} $T_F=P_{\Gamma}(\widehat{T}_F)$. Define
	$$
	\widehat{T}_F^{1}=\{(\eta, \xi) \in \widehat{T}_F, \eta <0 \}\quad \text{and} \quad \widehat{T}_F^{2}=\{(\eta, \xi) \in \widehat{T}_F, \eta >0 \}
	$$
	as two subelements of the element $\widehat{T}_F$. Thus, $T_F^1=P_{\Gamma}(\widehat{T}_F^1)$ and $T_F^2=P_{\Gamma}(\widehat{T}_F^2)$. For the interface element $T$, denote 
	$$\widehat{T}=R_{\Gamma}(T), \quad  \widehat{\Gamma}_{T_F}=R_{\Gamma}(\Gamma_{T_F}),$$
	where $\Gamma_{T_F}=\Gamma \cap T_F$. Obviously, for a mesh $\mathcal{T}_h$ with $h$, all the interface elements are inside the tubular neighborhood $N_{\Gamma}(h)$ of $\Gamma$. Hence, we can assume that the mesh size $h$ is sufficiently small such that $N_{\Gamma}(h) \subset N_{\Gamma}(\varepsilon)$, i.e., $h < \varepsilon$.

According to the well-known Frenet-Serret formulas for the tangent and normal vectors, we have
\begin{eqnarray*}
	\tau'(\xi)=-\kappa(\xi)\|\bg'(\xi)\|\bn(\xi),\quad
	\bn'(\xi)=\kappa(\xi)\|\bg'(\xi)\|\tau(\xi).
\end{eqnarray*}
Using the above equations, the Jacobian matrix of the Frenet transformation $P_{\Gamma}(\eta, \xi)$ is defined as
\begin{eqnarray*}
	\widehat{D}P_{\Gamma}(\eta, \xi)=\begin{pmatrix}
		\bn(\xi) & \bg'(\xi)+\eta \bn'(\xi)
	\end{pmatrix}=\begin{pmatrix}
	\bn(\xi) & \|\bg'(\xi)\|(1+\eta \kappa(\xi))\tau(\xi)
\end{pmatrix}.
\end{eqnarray*}
By the inverse function theorem, the Jacobian of the inverse mapping $R_{\Gamma}(x,y)$ has the following form:
\begin{eqnarray*}
	DR_{\Gamma}(x,y)=\begin{pmatrix}
		\bn(\xi)^T \\
		\|\bg'(\xi)\|^{-1} \psi(\eta,\xi)\tau(\xi)^T
	\end{pmatrix} \quad \text{with} \quad \psi(\eta,\xi)=(1+\eta \kappa(\xi))^{-1}, \quad \forall \, (x,y)\in N_{\Gamma}(\varepsilon).
\end{eqnarray*}

Assume that $\bg(\xi)$ is a regular $C^3([\xi_s, \xi_e],\Gamma)$ parametrization of $\Gamma$. This implies that $\|\bg'(\xi)\| \simeq 1$ and $\kappa(\xi) \in  C^1([\xi_s, \xi_e],\mathbb{R})$. Furthermore, we assume that $h \kappa_{\Gamma} \leqslant \dfrac{1}{2}$, where $\kappa_{\Gamma} =\max_{\xi \in [\xi_s, \xi_e]} |\kappa(\xi)|$, this ensures that $\psi \simeq |\widehat{D}P_{\Gamma}| \simeq |DR_{\Gamma}|\simeq 1$. The notation $a \simeq b$ is the equivalence relation $a \leqslant Cb$ and $b \leqslant Ca$ where $C$ is independent of the mesh size, the relative position of the interface and the viscosity coefficients $\beta_1$ and $\beta_2$.  

\section{The Numerical Scheme}
For the measurable domain $D \subset \mathbb{R}^2$, let $W^{k,p}(D)$ be the standard Sobolev spaces on $D$ associated with the norm $\|\cdot\|_{k,p,{D}}$. The corresponding Hilbert space is $H^k({D})=W^{k,2}({D})$ equipped with the norm $\|\cdot\|_{k,{D}}$. When $D = \Omega$ and $k=0$, we omit the subscript in the norm for simplicity.
Let
\begin{eqnarray*}
	PH^k({\Omega})=\{u: u|_{{\Omega}_s} \in H^k({\Omega}_s),\, s=1,2;\, \jump{u}_{\Gamma}=0 \,\, \text{and}\,\, \jump{\beta \nabla u \cdot \bn}_{\Gamma}=0 
	\}.
\end{eqnarray*}
The norm equipped with $PH^k({\Omega})$ is
$$
\|\cdot\|_{k,{\Omega}}^2=\|\cdot\|_{k,{\Omega}_1}^2+\|\cdot\|_{k,{\Omega}_2}^2.
$$

Let $\mathcal{T}_h$ be the shape-regular triangulation \cite{WGStokes1} of the domain $\Omega$, which is not necessarily aligned with the interface. Denote by $\mathcal{T}_h^n$ the set of non-interface elements and by $\mathcal{T}_h^I$ the set of interface elements, i.e., those elements intersected by $\Gamma$.  For
$T \in \mathcal{T}_h$, define the area and diameter of $T$ as $|T|$ and $h_T$, respectively. Set $h = \max_{T \in \mathcal{T}_h} h_T$. Denote by $\mathcal{E}_h$ the set of all edges in $\mathcal{T}_h$. Let $\mathcal{E}_h^b$ be the set of all edges lying on the boundary $\partial \Omega$. For $k \geqslant 1$, denote by $P_k(T)$ the space of polynomials of degree at most $k$ on $T$.

In addition to the standard shape-regular conditions, the interface elements are required to satisfy the following conditions:
\begin{itemize}
	\item refine the partition to satisfy $h_T < \delta_0$ such that $\overline{T} \subset U_{\delta_0}$ for all $ T \in \mathcal{T}_h^I$;
	\item the interface intersects the boundary of the element at exactly two points, and these two intersection points cannot lie on the same edge, including the two endpoints of that edge.
\end{itemize}

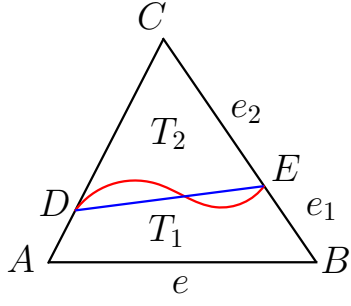
\begin{figure}
	\centering
	\ifpdf
	\setlength{\unitlength}{1bp}%
	  \begin{picture}(129.32, 121.53)(0,0)
		\put(0,0){\includegraphics{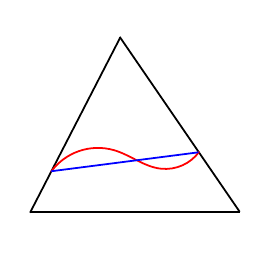}}
		\put(-1,16.92){\fontsize{14.23}{17.07}\selectfont $A$}
		\put(116.56,16.36){\fontsize{14.23}{17.07}\selectfont $B$}
		\put(47.91,107.75){\fontsize{14.23}{17.07}\selectfont $C$}
		\put(10.91,37.75){\fontsize{14.23}{17.07}\selectfont $D$}
		\put(97.91,50.75){\fontsize{14.23}{17.07}\selectfont $E$}
		\put(82.84,74.81){\fontsize{14.23}{17.07}\selectfont $e_2$}
		\put(111.05,38.24){\fontsize{14.23}{17.07}\selectfont $e_1$}
     	\put(52.84,64.81){\fontsize{14.23}{17.07}\selectfont $T_2$}
		\put(52.05,28.24){\fontsize{14.23}{17.07}\selectfont $T_1$}
		\put(60.90,8.73){\fontsize{14.23}{17.07}\selectfont $e$}
	\end{picture}%
	\caption{{The interface element.}}
	\label{The_interface_element}
\end{figure}
For the interface element $T$, set $T_1 = T \cap \Omega_1$ and $T_2 = T \cap \Omega_2$ while $\widetilde{T}_1$ and $\widetilde{T}_2$ are the quadrilateral $ABED$ and straight-edge triangle $CDE$, respectively.

\begin{lemma}\cite{IFE_2025}\label{Norm_equivalence}
	For the interface element $T$ as shown in Figure \ref{The_interface_element}, if $|C D | \geqslant \dfrac{1}{2}|AC|$ and $|CE | \geqslant  \dfrac{1}{2}|BC|$, there holds
		\begin{align}
		\|\cdot\|_{T_2} \simeq 	\|\cdot\|_{\widetilde{T}_2} \simeq \|\cdot\|_{T}, \quad {\rm{on}} \, P_1(T).
	\end{align}
	Otherwise, if  $|CD | \leqslant \dfrac{1}{2}|AC|$ or $|CE| \leqslant \dfrac{1}{2}|BC|$, then there holds 
	\begin{align}
	\|\cdot\|_{T_1} \simeq 	\|\cdot\|_{\widetilde{T}_1} \simeq \|\cdot\|_{T}, \quad {\rm{on}} \, P_1(T).
\end{align}
\end{lemma}
Next, we define the IFE functions on the interface element $T$. If $|CD | \geqslant \dfrac{1}{2}|AC|$ and $|CE | \geqslant \dfrac{1}{2}|BC|$, the definition of the IFE function $\varphi_i(x,y),\,i=1,2,3$ is given by
\begin{align}\label{IFE_function1}
	\varphi_i(x,y)=\left\{\begin{array}{lc}
			\varphi_2(x,y)=p_i(x,y) \in P_1(T_2), &\, \text{in}\, T_2\\				\varphi_{1}=\mathcal{E}^1(\varphi_{2})=\varphi_{2}-\left(1-\dfrac{\beta_2}{\beta_1}\right)\dfrac{\ell(x,y)}{\sqrt{\ell_x^2+\ell_y^2}}\nabla \varphi_{2} \cdot \bn, & \, \text{in}\, T_1
	\end{array}\right..
\end{align}
Otherwise, the IFE function is denoted by 
\begin{align}\label{IFE_function2}
	\varphi_i(x,y)=\left\{\begin{array}{lc}
		\varphi_1(x,y)=p_i(x,y) \in P_1(T_1), &\, \text{in}\, T_1\\				\varphi_{2}=\mathcal{E}^2(\varphi_{1})=\varphi_{1}-\left(1-\dfrac{\beta_1}{\beta_2}\right)\dfrac{\ell(x,y)}{\sqrt{\ell_x^2+\ell_y^2}}\nabla \varphi_{1} \cdot \bn, & \, \text{in}\, T_2
	\end{array}\right..
\end{align}
Here, we choose a set of basis functions of $P_1(T)$ as $\{p_i(x,y)\}_{i=1}^3$. It is easy to verify that the above function precisely satisfies the interface conditions \eqref{Interface_condition_1} - \eqref{Interface_condition_2}. The IFE space on an interface element $T$ is then defined as follows:
\begin{align*}
	\mathcal{V}_1(T)={\rm{span}}\{ \varphi_i(x,y),\, i=1,2,3\}.
\end{align*}
Thus, we define the following discontinuous finite element space on $\mathcal{T}_h^n$ and $\mathcal{T}_h^I$, respectively,
\begin{align*}
	V_h^n=&\{ v_h \in L^2(\Omega): v_h|_T\in P_1(T), \,  T \in \mathcal{T}_h^n\},\\
		V_h^I=&\{ v_h \in  L^2(\Omega): v_h|_T\in \mathcal{V}_1, \, T \in \mathcal{T}_h^I\}.
\end{align*}
Combining the space $V_h^n$ and $V_h^I$, the global finite element space is defined as
\begin{align*}
	V_{h} &=\{  v_h: v_h|_{\mathcal{T}_h^n} \in V_h^n, v_h|_{\mathcal{T}_h^I} \in V_h^I \},\\
	V_{h}^0 &= \{ v_h: v_h \in V_{h}, v_h|_e =0 \,  \text{on} \, e \in \mathcal{E}_h^b \}.
\end{align*}
We note that the functions in $V_h$ are discontinuous across the edge of the partition $\mathcal{T}_h$. Let $T_1$ and $T_2$ be two neighboring elements sharing a common edge $e$. We define the average and jump of a function $v$ on $e$ as follows:
\begin{align*}
	\{v\}_e=\dfrac{1}{2}(v|_{T_1}+v|_{T_2}),\quad\quad \jump{v}_e=v|_{T_1}-v|_{T_2}.
\end{align*}
If $e \subset \partial \Omega$, then the above definition is modified as 
\begin{align*}
	\{v\}_e=\jump{v}_e=v|_{T_1},\quad e=\partial T_1 \cap \partial \Omega.
\end{align*}

 For $T \in \mathcal{T}_h^n$, denote by $I_h$ the $L^2$ projection operator from $L^2(T)$ into $P_1(T)$. For the interface element $T \in \mathcal{T}_h^I$, let $P_h $ be the $L^2$ projection operator from $L^2(T)$ into $	\mathcal{V}_1(T)$.
 Let  $\Pi_h $ be the $L^2$ projection operator from $L^2(\Omega)$ into $V_h$ and satisfy the following definition:
 \begin{align*}
 	\Pi_h u=\left\{\begin{array}{cl}
 		I_hu,& T\in \mathcal{T}_h^n \\
 		P_h u,& T\in \mathcal{T}_h^I
 	\end{array}
 	\right..
 \end{align*}  
 Denote by $\bn_e$ the normal vector on the edge $e$. Based on the above definitions, we propose the following numerical scheme.

\begin{algorithm}[H]
	\caption{The Immersed SIPDG Scheme}
Find $u_h \in V_{h}$ and $u_h = \Pi_h g$ on $\partial \Omega$ to satisfy
	\begin{eqnarray}
		a_h(u_h, v_h)= L_h(v_h), \quad  \forall \, v_h \in V_{h}^0, \label{DGscheme}
	\end{eqnarray}
	where
\begin{align*}
	a_h(u,v)&=\sum_{T \in \mathcal{T}_h} (\beta \nabla u, \nabla v)_T-\sum_{e \in \mathcal{E}_h}\left( 
	\langle \jump{\beta \nabla u \cdot \bn_e}_e, \{v\}_e \rangle_{e}+\langle \{u\}_e,\jump{\beta \nabla v \cdot \bn_e}_e \rangle_e - \frac{ \sigma_0\gamma}{h} \langle \jump{u}_e, \jump{v}_e \rangle_e
	\right),\\
	 L_h(v_h)&=\sum_{T \in \mathcal{T}_h} (f,v)_T+\sum_{e \in \mathcal{E}_h^b}\langle -\beta \nabla v \cdot \bn_e+ \frac{\sigma_0\gamma}{h}v,g \rangle_e.
\end{align*}
\end{algorithm} 
Here, $\gamma=\dfrac{\beta_2^2}{\beta_1}$ and $\sigma_0 >0$ is a constant independent of the mesh size $h$,  the coefficients $\beta_1$ and $\beta_2$.

\section{The Approximation Capabilities of The IFE Space}
In this section, we investigate the approximation properties of the proposed IFE functions by taking the IFE function \eqref{IFE_function1} as an example. Let $\Pi_h^s: L^2(T_s) \rightarrow P_1(T_s),s= 1,2$ be the $L^2$ projection operator. That is, for $u_s \in H^2(\Omega_s)$, the projection $\Pi_h^s u_s$ satisfies the following equation:
\begin{eqnarray*}
	(\Pi_h^s u_s, v_h)_{T_s}=(u_s, v_h)_{T_s}, \quad \forall\, v_h \in P_1(T_s).
\end{eqnarray*}

\begin{lemma}\cite{Ciarlet_book_1978}\label{projection_error_in_Ti}
	Let $u_s \in H^2(\Omega_s)$, $s={1,2}$. Then for all integers $0\leqslant m \leqslant 2$, we have 
	\begin{eqnarray}
		\sum_{T \in \mathcal{T}_h^I} \left\|\Pi_h^s u_s- u_s \right\|_{m,T_F^s} \leqslant C h^{2-m}\|u_s\|_{2,\Omega_s}.
	\end{eqnarray}
\end{lemma}

\begin{lemma}\label{EstimatesInT2}
	For the interface element $T \in \mathcal{T}_h^I$ and $u \in PH^2(\Omega)$, then we have 
	\begin{eqnarray}\label{EstimatesInT2_1}
		\sum_{T \in \mathcal{T}_h^I}\|\Pi_h^1 u_1 - \mathcal{E}^1(\Pi_h^2 u_2)\|_{m, T \cap \Omega_{1}} \leqslant C \left(\dfrac{\beta_2}{\beta_1}\right) h^{2-m}\|u\|_{2,\Omega}.
	\end{eqnarray}
\end{lemma}
\begin{proof}
	For the interface element $T \in \mathcal{T}_h^I$, let $\omega=\Pi_h^1 u_1 - \mathcal{E}^1(\Pi_h^2 u_2)$. For the point $X=(x,y)^T \in T_1$, assume that both $AX$ and $BX$ do not intersect with the interface $\Gamma$, while $CX$ intersects with the interface $\Gamma$ at point $M(x_1,y_1)$.	
	
	\begin{figure}[H]
		\centering
		\setlength{\unitlength}{1bp}%
		\begin{picture}(129.32, 140.89)(0,0)
			\put(0,0){\includegraphics[scale=1.5]{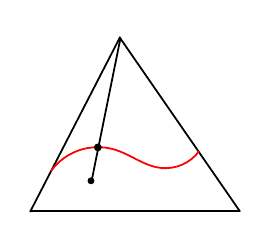}}
			\put(5.67,9.29){\fontsize{14.23}{17.07}\selectfont $A$}
			\put(176.56,10.73){\fontsize{14.23}{17.07}\selectfont $B$}
			\put(77.91,147.11){\fontsize{14.23}{17.07}\selectfont $C$}
			\put(70.37,35.42){\fontsize{10.23}{15.07}\selectfont $X(x,y)$}
			\put(110.37,35.42){\fontsize{14.23}{17.07}\selectfont $T_1$}
			\put(50.37,50.42){\fontsize{10.23}{15.07}\selectfont $\Gamma$}
			\put(75.09,70.65){\fontsize{10.23}{15.07}\selectfont $M(x_1,y_1)$}
    		\put(88.09,90.65){\fontsize{14.23}{17.07}\selectfont $T_2$}
		\end{picture}%
	\end{figure}	
	
Since $t \mapsto \omega(tx+(1-t)x_1,ty+(1-t)y_1)$ is $C^2$ function in terms of $t$, using the Newton-Leibniz formula and integration by parts yields
	\begin{eqnarray*}
		\begin{split}
			\omega(x,y)=&\omega(x_1,y_1)+\int_{0}^{1} \frac{d\omega(tx+(1-t)x_1,ty+(1-t)y_1)}{dt} dt\\
			=&\omega(x_1,y_1)+\omega_x(x_1,y_1)(x-x_1)+\omega_y(x_1,y_1)(y-y_1)\\
			&+\int_{0}^{1} (1-t)\frac{d^2\omega(tx+(1-t)x_1,ty+(1-t)y_1)}{dt^2} dt.		
		\end{split}
	\end{eqnarray*}
For the case $m=0$, the left-side term of the estimate \eqref{EstimatesInT2_1} is as follows
\begin{eqnarray*}
	\begin{split}
		\|\omega(x,y)\|_{T_1} \leqslant &	\|\omega(x_1,y_1)\|_{T_1} + \|\nabla \omega(x_1,y_1) \cdot \overrightarrow{MX}\|_{T_1}\\
		&+\left\|\int_{0}^{1}(1-t)\frac{d^2\omega(tx+(1-t)x_1,ty+(1-t)y_1)}{dt^2} dt\right\|_{T_1}.
	\end{split}
\end{eqnarray*}
Next, we estimate each term on the right side of the above inequality. 

(i) For the first term, by the interface condition \eqref{Interface_condition_1}, we obtain
\begin{eqnarray*}
	\begin{split}
		&\|\omega(x_1,y_1)\|_{T_1}^2\\
		=&\int_{T_1} \omega^2 (x_1,y_1)dT_1\\
		=&\int_{T_1} (\Pi_h^1 u_1(x_1,y_1) - \mathcal{E}^1(\Pi_h^2 u_2)(x_1,y_1))^2 dT_1\\
		=&\int_{T_1} (\Pi_h^1 u_1(x_1,y_1) - \Pi_h^2 u_2(x_1,y_1))^2 dT_1\\
		=&\int_{T_1} (\Pi_h^1 u_1(x_1,y_1) -u_1(x_1,y_1)+ u_2(x_1,y_1)-\Pi_h^2 u_2(x_1,y_1))^2 dT_1\\
		\leqslant & 2\int_{T_1} (\Pi_h^1 u_1(x_1,y_1)-u_1(x_1,y_1))^2  dT_1+2 \int_{T_1} (\Pi_h^2 u_2(x_1,y_1)-u_2(x_1,y_1))^2  dT_1.
	\end{split}
\end{eqnarray*}
Let $\widehat{\Pi_h^s u_s}(\eta, \xi)=\Pi_h^s u_s \circ P_{\Gamma} (\eta, \xi)$ and $\widehat{u}_s(\eta, \xi)= u_s\circ P_{\Gamma} (\eta, \xi)$, $s=1,2$. Using $(x_1,y_1)=P_{\Gamma}(0,\xi)$ leads to
\begin{eqnarray*}
	\Pi_h^s u_s (x_1,y_1)=\widehat{\Pi_h^s u_s}(0, \xi),\,\,  u_s(x_1,y_1)=\widehat{u}_s(0, \xi), \quad s=1,2.
\end{eqnarray*}
Let $\widehat{T}_1=R_{\Gamma}(T_1) \subset \widehat{T}^1_F$ and $\widehat{T}_2=R_{\Gamma}(T_2) \subset \widehat{T}^2_F$. It follows from $ |\widehat{D}P_{\Gamma}| \simeq |DP_{\Gamma}(x,y)|\simeq 1$, the trace inequality and Lemma \ref{projection_error_in_Ti} that
\begin{align*}
			&\|\omega(x_1,y_1)\|_{T_1}^2\\
			\leqslant & 2 \int_{\widehat{T}_1} \left(\widehat{\Pi_h^1 u_1} (0,\xi)-\widehat{u}_1(0,\xi) \right)^2 |DP_{\Gamma}| d\xi d \eta + 2 \int_{\widehat{T}_1} \left(\widehat{\Pi_h^2 u_2} (0,\xi)-\widehat{u}_2(0,\xi) \right)^2 |DP_{\Gamma}| d\xi d \eta\\
			\leqslant& 2\int_{-h}^{h}\int_{a_T}^{b_T}  \left(\widehat{\Pi_h^1 u_1} (0,\xi)-\widehat{u}_1(0,\xi) \right)^2 |DP_{\Gamma}| d\xi d \eta + 2\int_{-h}^{h}\int_{a_T}^{b_T}   \left(\widehat{\Pi_h^2 u_2} (0,\xi)-\widehat{u}_2(0,\xi) \right)^2 |DP_{\Gamma}|d\xi d \eta  \\
			\leqslant& Ch\int_{a_T}^{b_T} 	\left(\widehat{\Pi_h^1 u_1} (0,\xi)-\widehat{u}_1(0,\xi) \right)^2	d\xi +  Ch \int_{a_T}^{b_T} 	\left(\widehat{\Pi_h^2 u_2} (0,\xi)-\widehat{u}_2(0,\xi) \right)^2	d\xi\\
			\leqslant & Ch \|\widehat{\Pi_h^1 u_1}-\widehat{u}_1\|_{\widehat{\Gamma}_{T_F}}^2 + Ch \|\widehat{\Pi_h^2 u_2}-\widehat{u}_2\|_{\widehat{\Gamma}_{T_F}}^2\\
			\leqslant & Ch\left( h_{\widehat{T}_F}^{-1} \|\widehat{\Pi_h^1 u_1}-\widehat{u}_1\|_{\widehat{T}^1_F}^2+ h_{\widehat{T}_F} \|\widehat{\nabla}(\widehat{\Pi_h^1 u_1}-\widehat{u}_1)\|_{\widehat{T}^1_F}^2  \right)+Ch\left( h_{\widehat{T}_F}^{-1} \|\widehat{\Pi_h^2 u_2}-\widehat{u}_2\|_{\widehat{T}^2_F}^2+ h_{\widehat{T}_F} \|\widehat{\nabla}(\widehat{\Pi_h^2 u_2}-\widehat{u}_2)\|_{\widehat{T}^2_F}^2  \right)\\
			\leqslant & C \left( \| \Pi_h^1 u_1-u_1 \|_{T_F^1}^2  + h_T^2\|\nabla( \Pi_h^1 u_1-u_1)\|_{T_F^1}^2 \right)+C \left( \| \Pi_h^2 u_2-u_2 \|_{T_F^2}^2  + h_T^2\|\nabla( \Pi_h^2 u_2-u_2)\|_{T_F^2}^2 \right).
\end{align*}
Thus, we have
\begin{align*}
	\sum_{T \in \mathcal{T}_h^I} \|\omega(x_1,y_1)\|_{T_1}^2 \leqslant Ch^4(\|u_1\|_{2,\Omega_{1}}^2+\|u_2\|_{2,\Omega_{2}}^2).
\end{align*}

(ii) For the second term, set  $\overrightarrow{MX}=a\bn + b {\bm\tau} $ with $a=(\overrightarrow{MX},\bn)$ and $b=(\overrightarrow{MX},{\bm\tau})$. According to $|\overrightarrow{MX}|\leqslant Ch$ and the interface condition \eqref{Interface_condition_2}, we get
\begin{align*}
		&\|\beta_1 \nabla \omega(x_1,y_1) \cdot \overrightarrow{MX}\|_{T_1}^2\\
		=&\int_{T_1} \beta_1^2 \left(\nabla \omega(x_1,y_1) \cdot \overrightarrow{MX} \right)^2 dT_1\\
		=& \int_{T_1} (\beta_1 \left(\nabla \omega(x_1,y_1) \cdot (a\bn + b {\bm\tau})\right)^2 dT_1\\
		\leqslant& C h^2 \int_{T_1} (\beta_1 \nabla (\Pi_h^1 u_1(x_1,y_1) - \mathcal{E}^1(\Pi_h^2 u_2)(x_1,y_1)) \cdot \bn )^2dT_1\\
		&+C h^2 \int_{T_1} (\beta_1 \nabla (\Pi_h^1 u_1(x_1,y_1) - \mathcal{E}^1(\Pi_h^2 u_2)(x_1,y_1)) \cdot {\bm\tau} )^2dT_1\\
		\leqslant& Ch^2\int_{T_1} \left(
		\beta_1 \nabla (\Pi_h^1 u_1)(x_1,y_1)\cdot \bn - \beta_1 \nabla  u_1(x_1,y_1)\cdot \bn  - \beta_2 \nabla (\Pi_h^2 u_2)(x_1,y_1)\cdot \bn + \beta_2 \nabla u_2(x_1,y_1)\cdot \bn
		\right) dT_1\\
		&+ Ch^2\int_{T_1} \left(
		\beta_1 \nabla (\Pi_h^1 u_1)(x_1,y_1)\cdot {\bm\tau} - \beta_1 \nabla  u_1(x_1,y_1)\cdot {\bm\tau}  - \beta_1 \nabla (\Pi_h^2 u_2)(x_1,y_1)\cdot {\bm\tau} + \beta_1 \nabla u_2(x_1,y_1)\cdot {\bm\tau}
		\right) dT_1.
\end{align*}
Let $\widehat{\bn}=(\widehat{n}_1,\widehat{n}_2)=\bn \circ P_{\Gamma}(\eta, \xi)$. Using the chain rule leads to
\begin{align*}
	\nabla(\Pi_h^2 u_2) =&\begin{pmatrix}
		\dfrac{\partial \widehat{\Pi_h^2 u_2}}{\partial \xi} \dfrac{\partial \xi}{\partial x} + \dfrac{\partial \widehat{\Pi_h^2 u_2}}{\partial \eta} \dfrac{\partial \eta}{\partial x} \\[0.3cm]
	\dfrac{\partial \widehat{\Pi_h^2 u_2}}{\partial \xi} \dfrac{\partial \xi}{\partial y} + \dfrac{\partial \widehat{\Pi_h^2 u_2}}{\partial \eta} \dfrac{\partial \eta}{\partial y} 
	\end{pmatrix}=\begin{pmatrix}
	\dfrac{\partial \xi}{\partial x}&\dfrac{\partial \eta}{\partial x}\\[0.3cm]
	\dfrac{\partial \xi}{\partial y}&\dfrac{\partial \eta}{\partial y} 
\end{pmatrix} \widehat{\nabla}(\widehat{\Pi_h^2 u_2}) ,\\
\nabla(\Pi_h^1 u_1)=&\begin{pmatrix}
	\dfrac{\partial \widehat{\Pi_h^1 u_1}}{\partial \xi} \dfrac{\partial \xi}{\partial x} + \dfrac{\partial \widehat{\Pi_h^1 u_1}}{\partial \eta} \dfrac{\partial \eta}{\partial x} \\[0.3cm]
	\dfrac{\partial \widehat{\Pi_h^1 u_1}}{\partial \xi} \dfrac{\partial \xi}{\partial y} + \dfrac{\partial \widehat{\Pi_h^1 u_1}}{\partial \eta} \dfrac{\partial \eta}{\partial y}
\end{pmatrix}=\begin{pmatrix}
\dfrac{\partial \xi}{\partial x}&\dfrac{\partial \eta}{\partial x}\\[0.3cm]
\dfrac{\partial \xi}{\partial y}&\dfrac{\partial \eta}{\partial y} 
\end{pmatrix} \widehat{\nabla}(\widehat{\Pi_h^1 u_1}),\\
\nabla u_2=&\begin{pmatrix}
	\dfrac{\partial \widehat{u}_2}{\partial \xi} \dfrac{\partial \xi}{\partial x} +
	\dfrac{\partial \widehat{u}_2}{\partial \eta} \dfrac{\partial \eta}{\partial x}\\[0.3cm]
	\dfrac{\partial \widehat{u}_2}{\partial \xi} \dfrac{\partial \xi}{\partial y} + \dfrac{\partial \widehat{u}_2}{\partial \eta} \dfrac{\partial \eta}{\partial y} 
	\end{pmatrix}=\begin{pmatrix}
	\dfrac{\partial \xi}{\partial x}&\dfrac{\partial \eta}{\partial x}\\[0.3cm]
	\dfrac{\partial \xi}{\partial y}&\dfrac{\partial \eta}{\partial y} 
\end{pmatrix}\widehat{\nabla}\widehat{u}_2,\\
\nabla u_1=&\begin{pmatrix}
	\dfrac{\partial \widehat{u}_1}{\partial \xi} \dfrac{\partial \xi}{\partial x} +
	\dfrac{\partial \widehat{u}_1}{\partial \eta} \dfrac{\partial \eta}{\partial x}\\[0.3cm]
	\dfrac{\partial \widehat{u}_1}{\partial \xi} \dfrac{\partial \xi}{\partial y} + \dfrac{\partial \widehat{u}_1}{\partial \eta} \dfrac{\partial \eta}{\partial y} 
\end{pmatrix}=\begin{pmatrix}
\dfrac{\partial \xi}{\partial x}&\dfrac{\partial \eta}{\partial x}\\[0.3cm]
\dfrac{\partial \xi}{\partial y}&\dfrac{\partial \eta}{\partial y} 
\end{pmatrix} \widehat{\nabla}\widehat{u}_1.
\end{align*}
It follows from \cite[Theorem 2]{IFEHigh_2024}, the definition of $P_{\Gamma}$, and the trace inequality that
\begin{align*}
		&\int_{T_1} \left(
		\beta_1 \nabla (\Pi_h^1 u_1)(x_1,y_1)\cdot \bn - \beta_1 \nabla  u_1(x_1,y_1)\cdot \bn\right)^2 dT_1\\
		= & \int_{\widehat{T}_1} \left(\beta_1 \begin{pmatrix}
			\frac{\partial \xi}{\partial x}&\frac{\partial \eta}{\partial x}\\[0.1cm]
			\frac{\partial \xi}{\partial y}&\frac{\partial \eta}{\partial y} 
		\end{pmatrix} \left(\widehat{\nabla}(\widehat{\Pi_h^1 u_1})-\widehat{\nabla}(\widehat{u_1}) \right)\cdot \widehat{\bn} \right)^2 |DP_{\Gamma}|d\xi d\eta\\
\leqslant& C \int_{\widehat{T}_1} \beta_1^2 \left(\widehat{\nabla}(\widehat{\Pi_h^1 u_1})(0,\xi) \cdot \bn - \widehat{\nabla}\widehat{u}_1(0,\xi) \cdot \bn \right)^2 d\xi d\eta\\
\leqslant & Ch\int_{a_T}^{b_T} \beta_1^2 \left(\widehat{\nabla}(\widehat{\Pi_h^1 u_1})(0,\xi) \cdot \bn - \widehat{\nabla}\widehat{u}_1(0,\xi) \cdot \bn \right)^2 d\xi \\
\leqslant& C\beta_1^2 h\| \widehat{\nabla}(\widehat{\Pi_h^1 u_1}-\widehat{u}_1) \|_{\widehat{\Gamma}_{T_F}}^2\\
\leqslant& C\beta_1^2h\left(h_{\widehat{T}_F}^{-1}\| \widehat{\nabla}(\widehat{\Pi_h^1 u_1}-\widehat{u}_1) \|_{\widehat{T}_F^1}^2+h_{\widehat{T}_F}\|\widehat{\nabla}(\widehat{\nabla}(\widehat{\Pi_h^1 u_1}-\widehat{u}_1))\|_{\widehat{T}_F^1}^2\right)\\
\leqslant&C\beta_1^2(\|\nabla(\Pi_h^1 u_1-u_1)\|_{T_F^1}^2+h^2\|\nabla(\nabla(\Pi_h^1 u_1-u_1))\|_{T_F^1}^2).
\end{align*}
Thus, we derive
\begin{align*}
	\sum_{T \in \mathcal{T}_h^I} \int_{T_1} \left(
	\beta_1 \nabla (\Pi_h^1 u_1)(x_1,y_1)\cdot \bn - \beta_1 \nabla  u_1(x_1,y_1)\cdot \bn\right)^2 dT_1  \leqslant Ch^2\beta_1^2\|u_1\|_{2,\Omega_{1}}^2.
\end{align*}
Similarly, we have
\begin{align*}
	\sum_{T \in \mathcal{T}_h^I}\int_{T_1} \left(
	\beta_2 \nabla (\Pi_h^2 u_2)(x_1,y_1)\cdot \bn - \beta_2 \nabla  u_2(x_1,y_1)\cdot \bn\right)^2 dT_1 \leqslant& Ch^2\beta_2^2 \|u_2\|_{2,\Omega_{2}}^2,\\
	\sum_{T \in \mathcal{T}_h^I}\int_{T_1} \left(
	\beta_1 \nabla (\Pi_h^1 u_1)(x_1,y_1)\cdot {\bm\tau} - \beta_1 \nabla  u_1(x_1,y_1)\cdot {\bm\tau} \right)^2 dT_1 \leqslant& Ch^2\beta_1^2 \|u_1\|_{2,\Omega_{1}}^2,\\
	\sum_{T \in \mathcal{T}_h^I}\int_{T_1} \left(
	\beta_1\nabla (\Pi_h^2 u_2)(x_1,y_1)\cdot {\bm\tau} - \beta_1 \nabla  u_2(x_1,y_1)\cdot {\bm\tau} \right)^2 dT_1 \leqslant& Ch^2\beta_1^2 \|u_2\|_{2,\Omega_{2}}^2.
\end{align*}
Accordingly, we obtain 
\begin{eqnarray*}
	\sum_{T \in \mathcal{T}_h^I}\|\nabla \omega \cdot \overrightarrow{MX} \|_{T_1}^2 \leqslant  Ch^4\left( \left(\dfrac{\beta_2}{\beta_1}\right)^2 \|u_1\|_{2,\Omega_{1}}^2+\|u_2\|_{2,\Omega_{2}}^2\right).
\end{eqnarray*}

(iii) For the third term, using the chain rule yields
\begin{align*}
		&\int_{0}^{1} (1-t)\frac{d^2\omega(tx+(1-t)x_1,ty+(1-t)y_1)}{dt^2} dt\\
		=&\int_{0}^{1} (1-t)(\omega_{xx}(x-x_1)^2+2\omega_{xy}(x-x_1)(y-y_1)+\omega_{yy}(y-y_1)^2) dt \\
		\leqslant& Ch^2 \int_{0}^{1} (1-t)(\omega_{xx}+2\omega_{xy}+\omega_{yy}) dt.
\end{align*}
By the definition of IFE function \eqref{IFE_function1}, the follow functions 
\begin{align*}
		\omega_{xx}=&\left(1-\dfrac{\beta_2}{\beta_1}\right) \left(
		\dfrac{\ell(x,y)}{\sqrt{\ell_x^2+\ell_y^2}}\right)_{xx}\left(\nabla(\Pi_h^2 u_2)\cdot \bn\right)+2\left(1-\dfrac{\beta_2}{\beta_1}\right)\left(
		\dfrac{\ell(x,y)}{\sqrt{\ell_x^2+\ell_y^2}}\right)_{x}\left(\nabla(\Pi_h^2 u_2)\cdot \bn_x\right)\\
		&+\left(1-\dfrac{\beta_2}{\beta_1}\right)\dfrac{\ell(x,y)}{\sqrt{\ell_x^2+\ell_y^2}}\left(\nabla(\Pi_h^2 u_2)\cdot \bn_{xx}\right) ,\\
	\omega_{xy}=&\left(1-\dfrac{\beta_2}{\beta_1}\right)\left(
    \dfrac{\ell(x,y)}{\sqrt{\ell_x^2+\ell_y^2}}\right)_{xy}(\nabla(\Pi_h^2 u_2)\cdot \bn)+\left(1-\dfrac{\beta_2}{\beta_1}\right)\left(
    \dfrac{\ell(x,y)}{\sqrt{\ell_x^2+\ell_y^2}}\right)_{x}(\nabla(\Pi_h^2 u_2)\cdot \bn_y)\\
    &+\left(1-\dfrac{\beta_2}{\beta_1}\right)\left(
    \dfrac{\ell(x,y)}{\sqrt{\ell_x^2+\ell_y^2}}\right)_{y}(\nabla(\Pi_h^2 u_2)\cdot \bn_x)+\left(1-\dfrac{\beta_2}{\beta_1}\right)\left(
    \dfrac{\ell(x,y)}{\sqrt{\ell_x^2+\ell_y^2}}\right)(\nabla(\Pi_h^2 u_2)\cdot \bn_{xy}),\\	
    \omega_{yy}=&\left(1-\dfrac{\beta_2}{\beta_1}\right)\left(
    \dfrac{\ell(x,y)}{\sqrt{\ell_x^2+\ell_y^2}}\right)_{yy}\left(\nabla(\Pi_h^2 u_2)\cdot \bn\right)+2\left(1-\dfrac{\beta_2}{\beta_1}\right)\left(
    \dfrac{\ell(x,y)}{\sqrt{\ell_x^2+\ell_y^2}}\right)_{y}\left(\nabla(\Pi_h^2 u_2)\cdot \bn_{y}\right)\\
    &+\left(1-\dfrac{\beta_2}{\beta_1}\right)\dfrac{\ell(x,y)}{\sqrt{\ell_x^2+\ell_y^2}}\left(\nabla(\Pi_h^2 u_2)\cdot \bn_{yy} \right)
\end{align*}
are bounded. Thus, we have
\begin{align*}
		&\sum_{T \in \mathcal{T}_h^I}\left\|\int_{0}^{1} (1-t)\frac{d^2\omega(tx+(1-t)x_1,ty+(1-t)y_1)}{dt^2} dt\right\|_{T_1}^2\\
		=&\sum_{T \in \mathcal{T}_h^I}\int_{T_1} \left(\int_{0}^{1} (1-t)\frac{d^2\omega(tx+(1-t)x_1,ty+(1-t)y_1)}{dt^2} dt\right)^2 dT_1\\
		\leqslant& Ch^4 \sum_{T \in \mathcal{T}_h^I}\int_{T_1}  (\omega_{xx}+2\omega_{xy}+\omega_{yy})^2 dT_1\\
		\leqslant& Ch^4\left( \|u_1\|_{1,\Omega_{1}}^2+ \|u_2\|_{2,\Omega_{2}}^2 \right).
\end{align*}
To sum up, we obtain 
\begin{eqnarray*}
   \sum_{T \in \mathcal{T}_h^I}	\|\Pi_h^1 u_1 - \mathcal{E}^1(\Pi_h^2 u_2)\|_{T_1} \leqslant C\left(\dfrac{\beta_2}{\beta_1}\right) h^{2}\left( \|u_1\|_{1,\Omega_{1}}^2+ \|u_2\|_{2,\Omega_{2}}^2 \right) .
\end{eqnarray*}

For the case $m=1$, using the Newton-Leibniz formula yields
\begin{eqnarray*}
	\begin{split}
		\omega_x(x,y)=&\omega_x(x_1,y_1)+\int_{0}^{1} \frac{d\omega_x(tx+(1-t)x_1,ty+(1-t)y_1)}{dt} dt
	\end{split}
\end{eqnarray*}
and 
\begin{eqnarray*}
	\begin{split}
		\omega_y(x,y)=&\omega_y(x_1,y_1)+\int_{0}^{1} \frac{d\omega_y(tx+(1-t)x_1,ty+(1-t)y_1)}{dt} dt.
	\end{split}
\end{eqnarray*}
Then, we have
\begin{eqnarray*}
	\begin{split}
		\|\nabla \omega \|_{T_1}^2=&\int_{T_1} (\omega_x^2(x,y)+\omega_y^2(x,y)) dT_1\\
		=&\int_{T_1} \left(\omega_x(x_1,y_1)+\int_{0}^{1} \frac{d\omega_x(tx+(1-t)x_1,ty+(1-t)y_1)}{dt} dt \right)^2 dT_1\\
		 &+ \int_{T_1} \left(\omega_y(x_1,y_1)+\int_{0}^{1} \frac{d\omega_y(tx+(1-t)x_1,ty+(1-t)y_1)}{dt} dt \right)^2 dT_1\\
		 =&I_1+I_2+I_3,
	\end{split}
\end{eqnarray*}
where
\begin{align*}
	I_1=&\int_{T_1} \left(\omega_x^2(x_1,y_1)+\omega_y^2(x_1,y_1) \right) dT_1,\\
	I_2=&\int_{T_1} \left( 2\omega_x(x_1,y_1)\int_{0}^{1} \frac{d\omega_x(tx+(1-t)x_1,ty+(1-t)y_1)}{dt} dt \right)  dT_1\\
	&+ \int_{T_1} \left( 2\omega_y(x_1,y_1)\int_{0}^{1} \frac{d\omega_y(tx+(1-t)x_1,ty+(1-t)y_1)}{dt} dt \right)  dT_1,\\
	I_3=&\int_{T_1} \left( \int_{0}^{1} \frac{d\omega_x(tx+(1-t)x_1,ty+(1-t)y_1)}{dt} dt  \right)^2 dT_1\\
	&+ \int_{T_1}\left( \int_{0}^{1} \frac{d\omega_y(tx+(1-t)x_1,ty+(1-t)y_1)}{dt} dt  \right)^2 dT_1.
\end{align*}

(i) For the first term $I_1$, according to the interface conditions \eqref{Interface_condition_2} and Eq.\eqref{nablau2nablau1}, we obtain
\begin{align*}
	&\nabla \omega(x_1,y_1)\\
	=&\nabla(\Pi_h^1u_1(x_1,y_1)-\mathcal{E}_h^1(\Pi_h^2u_2)(x_1,y_1))\\
	=&\nabla(\Pi_h^1u_1(x_1,y_1)) -\nabla u_1(x_1,y_1)+\nabla u_1(x_1,y_1)-\nabla(\mathcal{E}_h^1(\Pi_h^2u_2)(x_1,y_1))\\
	=&\nabla(\Pi_h^1u_1(x_1,y_1)) -\nabla u_1(x_1,y_1)+N_2(x_1,y_1)\nabla u_2(x_1,y_1)-N_2(x_1,y_1) \nabla(\Pi_h^2u_2(x_1,y_1)).
\end{align*}
Therefore, using the estimate \eqref{Estimate_N} of $N_2(X)$ leads to
\begin{align*}
	&\|I_1\|_{T_1}^2=\|\nabla \omega(x_1,y_1) \|_{T_1}^2 \leqslant C \|\nabla(\Pi_h^1u_1 -u_1)(x_1,y_1)\|_{T_1}^2+ C\left(\dfrac{\beta_2}{\beta_1}\right)^2 \|\nabla(\Pi_h^2u_2 -u_2)(x_1,y_1)\|_{T_1}^2.
\end{align*}
It follows from the chain rule, the trace inequality and Lemma \ref{projection_error_in_Ti} that
\begin{align*}
	&\|\nabla(\Pi_h^2u_2 -u_2)(x_1,y_1)\|_{T_1}^2\\
	\leqslant &C \int_{-h}^{h}  \int_{a_T}^{b_T} \left( \widehat{\nabla}(\widehat{\Pi_h^2 u_2})(0,\xi) - \widehat{\nabla}\widehat{u}_2(0,\xi)\right)^2  d\xi d\eta\\
	\leqslant &Ch \|\widehat{\nabla}(\widehat{\Pi_h^2 u_2}-\widehat{u}_2)\|_{\widehat{\Gamma}_{T_F}}^2\\
	\leqslant& Ch\left( h_{\widehat{T}_F}^{-1}\|\widehat{\nabla}(\widehat{\Pi_h^2 u_2}-\widehat{u}_2)\|_{\widehat{T}_F^2}^2+h_{\widehat{T}_F}\|\widehat{\nabla}(\widehat{\nabla}(\widehat{\Pi_h^2 u_2}-\widehat{u}_2))\|_{\widehat{T}_F^2}^2 \right)\\
	\leqslant & C\left( \|\nabla(\Pi_h^2 u_2-u_2) \|_{T_F^2}^2+h^2 \|\nabla(\nabla(\Pi_h^2 u_2-u_2))\|_{T_F^2}^2 \right).
\end{align*}
Similarly, we get 
\begin{align*}
	\|\nabla(\Pi_h^1u_1 -u_1)(x_1,y_1)\|_{T_1}^2\leqslant C\left( \|\nabla(\Pi_h^1 u_1-u_1) \|_{T_F^1}^2+h^2 \|\nabla(\nabla(\Pi_h^1 u_1-u_1))\|_{T_F^1}^2 \right).
\end{align*}
Therefore, the estimate of $I_1$ is as follows:
\begin{align*}
 \sum_{T \in \mathcal{T}_h^I} \|I_1\|_{T_1}^2 \leqslant  C \left(\dfrac{\beta_2}{\beta_1}\right)^2 h^2(\|u_1\|_{2,\Omega_1}^2+\|u_2\|_{2,\Omega_2}^2).
\end{align*}

(ii) For the third term, we have
\begin{align*}
	    \sum_{T \in \mathcal{T}_h^I}\|I_3\|_{T_1}^2=&\sum_{T \in \mathcal{T}_h^I}\int_{T_1} \left(\int_{0}^{1} \left( \omega_{xx}(x-x_1)+\omega_{xy}(y-y_1)\right)dt \right)^2 dT_1\\
		&+\sum_{T \in \mathcal{T}_h^I}\int_{T_1} \left(\int_{0}^{1} \left( \omega_{yx}(x-x_1)+\omega_{yy}(y-y_1)\right)dt \right)^2 dT_1\\
		\leqslant&C h^2\sum_{T \in \mathcal{T}_h^I}\int_{T_1} \int_{0}^{1} \left( 2\omega_{xx}^2+4\omega_{xy}^2+2\omega_{yy}^2\right) dt dT_1\\
		\leqslant&Ch^2(\|u_1\|_{2,\Omega_1}^2+\|u_2\|_{2,\Omega_2}^2).
\end{align*}
Combining the above two estimates and Young inequality, we get
\begin{eqnarray*}
	   \sum_{T \in \mathcal{T}_h^I}\|I_2\|_{T_1}^2\leqslant C \sum_{T \in \mathcal{T}_h^I}\|I_1\|_{T_1}^2+ C\sum_{T \in \mathcal{T}_h^I}\|I_2\|_{T_1}^2\leqslant C \left(\dfrac{\beta_2}{\beta_1}\right)^2 h^2(\|u_1\|_{2,\Omega_1}^2+\|u_2\|_{2,\Omega_2}^2).
\end{eqnarray*}
To sum up, we obtain
\begin{eqnarray*}
\sum_{T \in \mathcal{T}_h^I}	\|\nabla \omega \|_{T_1}^2\leqslant C \left(\dfrac{\beta_2}{\beta_1}\right)^2 h^2(\|u_1\|_{2,\Omega_1}^2+\|u_2\|_{2,\Omega_2}^2).
\end{eqnarray*}
Finally, using \cite[Lemma 15]{IFE_2025} yields 
\begin{align*}
	\sum_{T \in \mathcal{T}_h^I}\|\Pi_h^1 u_1 - \mathcal{E}^1(\Pi_h^2 u_2)\|_{m, T_1} \leqslant C \left(\dfrac{\beta_2}{\beta_1}\right) h^{2-m}\|u\|_{2,\Omega}.
\end{align*}
The proof of the above lemma is completed.
\end{proof}

\begin{theorem}\label{Projection_estimate_global}
	 Assume $u \in PH^2(\Omega)$ satisfies the problems \eqref{Elliptic_problem} - \eqref{Interface_condition_2}, then we have 
	\begin{eqnarray*}
		\sum_{T \in \mathcal{T}_h} \left\|\Pi_h u- u \right\|_{m,T} \leqslant C\left(\dfrac{\beta_2}{\beta_1}\right) h^{2-m}\|u\|_{2,\Omega}.
	\end{eqnarray*}
\end{theorem}
\begin{proof}
		For the non-interface element $T \in \mathcal{T}_h^n$, it follows from the projection estimate that
		\begin{align*}
			\sum_{T \in \mathcal{T}_h^n}\left\|I_h u- u \right\|_{m,T} \leqslant Ch^{2-m}\|u\|_{2,\Omega}.
		\end{align*}
	For the interface element $T \in \mathcal{T}_h^I$ and any $q \in \mathcal{V}_1(T)$, we obtain
	\begin{eqnarray*}
	\sum_{T \in \mathcal{T}_h^I}\left\|P_h u- u \right\|_{m,T} \leqslant C \sum_{T \in \mathcal{T}_h^I}	\left\| u-q \right\|_{m,T}.
	\end{eqnarray*}
Taking $q=(\mathcal{E}^1(\Pi_h^2 u_2),\Pi_h^2 u_2)$ and using Lemma \ref{projection_error_in_Ti} and Lemma \ref{EstimatesInT2} leads to
\begin{eqnarray*}
	\begin{split}
		\sum_{T \in \mathcal{T}_h^I}\left\| u-q \right\|_{m,T} \leqslant& \sum_{T \in \mathcal{T}_h^I} (\left\| u_1 -\mathcal{E}^1(\Pi_h^2 u_2) \right\|_{m,T \cap \Omega_1}+\left\| u_2 -\Pi_h^2 u_2 \right\|_{m,T \cap \Omega_2})\\
		\leqslant& \sum_{T \in \mathcal{T}_h^I} (\left\| u_1 -\Pi_h^1 u_1 \right\|_{m,T_1}+\left\|\Pi_h^1 u_1 -\mathcal{E}^1(\Pi_h^2 u_2) \right\|_{m,T \cap \Omega_1}+\left\| u_2 -\Pi_h^2 u_2 \right\|_{m,T_2})\\
		\leqslant& C\left(\dfrac{\beta_2}{\beta_1}\right)h^{2-m}\|u\|_{2,\Omega}.
	\end{split}
\end{eqnarray*}
The proof of the above theorem is completed.
\end{proof}

\begin{remark}
Similar to the analysis in Lemma \ref{EstimatesInT2} and Theorem \ref{Projection_estimate_global} , we derive the following estimate for the IFE function \eqref{IFE_function2}:
	\begin{eqnarray*}
		\sum_{T \in \mathcal{T}_h} \left\|\Pi_h u- u \right\|_{m,T} \leqslant C h^{2-m}\|u\|_{2,\Omega}.
\end{eqnarray*}
\end{remark}

\section{The Important Inequality}
In this section, we derive an important inequality for the error analysis. The interface element intersects with the interface in two patterns.

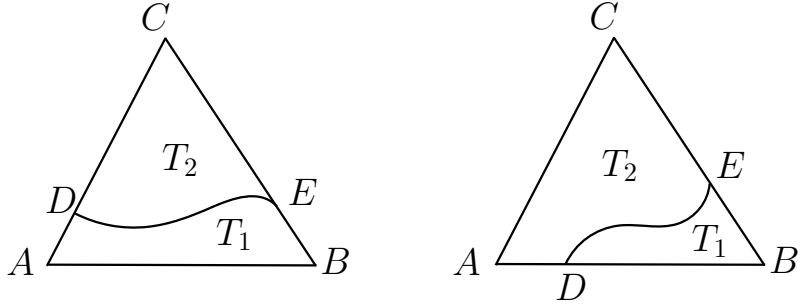
\begin{figure}[H]
	\centering
	\setlength{\unitlength}{1bp}%
	\begin{picture}(299.30, 127.10)(0,0)
		\put(0,0){\includegraphics{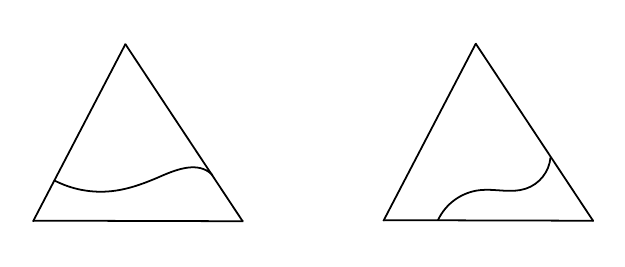}}
		\put(0.67,18.14){\fontsize{14.23}{17.07}\selectfont $A$}
		\put(51.03,110.08){\fontsize{14.23}{17.07}\selectfont $C$}
		\put(118.36,18.62){\fontsize{14.23}{17.07}\selectfont $B$}
		\put(15.08,41.31){\fontsize{14.23}{17.07}\selectfont $D$}
		\put(106.05,44.44){\fontsize{14.23}{17.07}\selectfont $E$}
		\put(167.85,18.38){\fontsize{14.23}{17.07}\selectfont $A$}
		\put(219.22,110.32){\fontsize{14.23}{17.07}\selectfont $C$}
		\put(286.54,18.87){\fontsize{14.23}{17.07}\selectfont $B$}
		\put(206.33,8.73){\fontsize{14.23}{17.07}\selectfont $D$}
		\put(266.42,54.58){\fontsize{14.23}{17.07}\selectfont $E$}
		\put(79.25,29.58){\fontsize{14.23}{17.07}\selectfont $T_1$}
		\put(58.98,57.57){\fontsize{14.23}{17.07}\selectfont $T_2$}
		\put(257.57,27.16){\fontsize{14.23}{17.07}\selectfont $T_1$}
		\put(223.79,55.39){\fontsize{14.23}{17.07}\selectfont $T_2$}
	\end{picture}%
	\caption{Two intersection patterns between the interface and an interface element. Left: Type 1 element, Right: Type 2 element.}
	\label{Two_patterns_of_interface}
\end{figure}
Next, we prove the following inequality using the Type 1 element as an example. A similar argument shows that this inequality also holds for the Type 2 element.
\begin{lemma}\label{Estimate_interface_function}
	For the level-set function $\ell(x,y)$ and the interface element $T \in \mathcal{T}_h^I$, the following estimates hold true:
	\begin{align*}
		&\left\|\dfrac{\ell(x,y)\ell_x(x,y)}{\ell_x^2(x,y)+\ell_y^2(x,y)}\right\|_{\infty, T} \leqslant Ch,\quad \left\|\dfrac{\ell(x,y)\ell_y(x,y)}{\ell_x^2(x,y)+\ell_y^2(x,y)}\right\|_{\infty, T} \leqslant Ch ,\\
			&\left\|\nabla(\dfrac{\ell(x,y)\ell_x(x,y)}{\ell_x^2(x,y)+\ell_y^2(x,y)})\right\|_{\infty, T} \leqslant C,\quad \left\|\nabla(\dfrac{\ell(x,y)\ell_y(x,y)}{\ell_x^2(x,y)+\ell_y^2(x,y)})\right\|_{\infty, T} \leqslant C,\\
			&\left\| \nabla\left(\nabla\left(\dfrac{\ell\ell_x}{\ell_x^2+\ell_y^2}\right) \right)\right\|_{\infty, T}^2 \leqslant C, \quad \left\| \nabla\left(\nabla\left(\dfrac{\ell\ell_y}{\ell_x^2+\ell_y^2}\right) \right)\right\|_{\infty, T}^2 \leqslant C.
	\end{align*}
\end{lemma}

\begin{theorem}
For the interface element $T \in \mathcal{T}_h^I$, we have
\begin{eqnarray}\label{Trace_inverse_inequality_interface}
	\|\beta \nabla \varphi\|_{\partial T} \leqslant C \dfrac{\beta_2}{\sqrt{\beta_1}}h^{-\tfrac{1}{2}}\|\sqrt{\beta}\nabla \varphi \|_{T},\quad  \forall \, \varphi \in \mathcal{V}_1(T).
\end{eqnarray}	
\end{theorem}
\begin{proof}
	Without loss of generality, we consider the non-interface edge $AB$ and the interface edge $AC$, see Figure \ref{The_interface_element}. 
	
	First, if $|CD | \geqslant \dfrac{1}{2}|AC|$ and $|CE | \geqslant \dfrac{1}{2}|BC|$. For the interface edge $AC$, using the triangle inequality yields
	\begin{align}\label{Trace_proof_1}
		\|\beta \nabla \varphi\|_{L^2(AC)}\leqslant \|\beta_2 \nabla \varphi_2\|_{L^2(CD)}+\|\beta_1 \nabla \varphi_1\|_{L^2(AD)}.
	\end{align}
	For $\|\beta_2 \nabla \varphi_2\|_{L^2(CD)}$, according to the trace inequality, the inverse inequality and Lemma \ref{Norm_equivalence}, we obtain
	\begin{align}\label{Trace_proof_5}
	\|\beta_2 \nabla \varphi_2\|_{L^2(CD)}\leqslant C h_T^{-\tfrac{1}{2}}\|\beta_2 \nabla \varphi_2\|_{L^2(CDE)}\leqslant C h_T^{-\tfrac{1}{2}}\|\beta_2 \nabla \varphi_2\|_{T_2}.
	\end{align}
	For $\|\beta_1 \nabla \varphi_1\|_{L^2(AD)}$, by the trace inequality, we get
	\begin{align}\label{Trace_proof_3}
		\begin{split}
			\|\beta_1 \nabla \varphi_1\|_{L^2(AD)}\leqslant& \|\beta_1 \nabla \varphi_1\|_{L^2(AC)}\\
			\leqslant&Ch_T^{-\tfrac{1}{2}}\|\beta_1 \nabla \varphi_1\|_{T}+Ch_T^{\tfrac{1}{2}}\|\nabla(\beta_1 \nabla \varphi_1)\|_{T}\\
			\leqslant& Ch_T^{-\tfrac{1}{2}}(\|\beta_1 \nabla \varphi_1\|_{T_1}+\|\beta_1 \nabla \varphi_1\|_{T_2})+Ch_T^{\tfrac{1}{2}}\|\nabla(\beta_1 \nabla \varphi_1)\|_{T}.
		\end{split}
	\end{align}
For $\|\nabla(\beta_1 \nabla \varphi_1)\|_{T}$, it follows from the definition of the function $\varphi_{1}$ and $\varphi_2 \in P_1(T_2)$ that
\begin{align*}
		&\|\nabla(\beta_1\nabla \varphi_{1})\|_{T}^2\\
		\leqslant&C \|\nabla(\beta_1\nabla \varphi_{2})\|_{T}^2+C |\beta_1-\beta_2|\left\|\nabla\varphi_{2,x} \cdot \nabla\left(\dfrac{\ell\ell_x}{\ell_x^2+\ell_y^2} \right)\right\|_{T}^2 \\
		&+C |\beta_1-\beta_2| \left\|\dfrac{\ell\ell_x}{\ell_x^2+\ell_y^2} \nabla(\nabla\varphi_{2,x} )\right\|_{T}^2+C |\beta_1-\beta_2|\left\|\varphi_{2,x} \nabla\left(\nabla\left(\dfrac{\ell\ell_x}{\ell_x^2+\ell_y^2}\right) \right)\right\|_{T}^2\\
		&+C |\beta_1-\beta_2|\left\|\nabla\varphi_{2,y} \cdot \nabla\left(\dfrac{\ell\ell_y}{\ell_x^2+\ell_y^2} \right)\right\|_{T}^2+C |\beta_1-\beta_2| \left\|\dfrac{\ell\ell_y}{\ell_x^2+\ell_y^2} \nabla(\nabla\varphi_{2,y} )\right\|_{T}^2\\
		&+C |\beta_1-\beta_2|\left\|\varphi_{2,y} \nabla\left(\nabla\left(\dfrac{\ell\ell_y}{\ell_x^2+\ell_y^2}\right) \right)\right\|_{T}^2\\
		\leqslant& C |\beta_1-\beta_2|\left\|\varphi_{2,x} \nabla\left(\nabla\left(\dfrac{\ell\ell_x}{\ell_x^2+\ell_y^2}\right) \right)\right\|_{T}^2+ C |\beta_1-\beta_2|\left\|\varphi_{2,y} \nabla\left(\nabla\left(\dfrac{\ell\ell_y}{\ell_x^2+\ell_y^2}\right) \right)\right\|_{T}^2\\
		\leqslant&I_1+I_2.
\end{align*}
For the term $I_1$, by 
$\beta_2 \geqslant \beta_1 >0$ and Lemma \ref{Estimate_interface_function}, we have
\begin{align*}
	\begin{split}
		I_1=&C |\beta_1-\beta_2|\left\|\varphi_{2,x} \nabla\left(\nabla\left(\dfrac{\ell\ell_x}{\ell_x^2+\ell_y^2}\right) \right)\right\|_{T}^2\\
		\leqslant& C \|\beta_2\nabla \varphi_2\|_T^2 \left\| \nabla\left(\nabla\left(\dfrac{\ell\ell_x}{\ell_x^2+\ell_y^2}\right) \right)\right\|_{\infty, T}^2\\
		\leqslant& C h_T^{-2}  \|\beta_2\nabla \varphi_2\|_T^2.
	\end{split}
\end{align*}
Similarly, the term $I_2$ has the following estimate
\begin{align*}
	I_2\leqslant& C h_T^{-2}  \|\beta_2\nabla \varphi_2\|_T^2.
\end{align*}
Combing with Lemma \ref{Norm_equivalence}, we obtain
\begin{align}\label{Trace_proof_2}
	\|\nabla(\beta_1\nabla \varphi_{1})\|_{T}^2\leqslant Ch_T^{-2}\|\beta_2\nabla \varphi_2\|_T^2\leqslant Ch_T^{-2}\|\beta_2\nabla \varphi_2\|_{T_2}^2\leqslant Ch_T^{-2}\|\beta\nabla \varphi\|_{T}^2.
\end{align}
For $\|\beta_1 \nabla \varphi_1\|_{T_2}$, using $\beta_2 \geqslant \beta_1 >0$, the definition of $\varphi_1$, the inverse inequality and Lemma \ref{Estimate_interface_function} leads to
\begin{align}\label{Trace_proof_4}
	\begin{split}
		\|\beta_1 \nabla \varphi_1\|_{T_2}\leqslant&\|\beta_1 \nabla(\varphi_1-\varphi_2) \|_{T_2}+\|\beta_1 \nabla \varphi_2\|_{T_2}\\
		\leqslant&\|\beta_2 \nabla \varphi_2\|_{T_2}+\left\|\beta_1\nabla\left(\left(1-\frac{\beta_2}{\beta_1}\right)\frac{\ell(x,y)}{\sqrt{\ell_x^2+\ell_y^2}}\nabla \varphi_2 \cdot \bn
		\right) \right\|_{T_2}\\
		\leqslant&\|\beta_2 \nabla \varphi_2\|_{T_2}+\left\|\beta_2\nabla\left(
		\frac{\ell\ell_x}{\ell_x^2+\ell_y^2}\varphi_{2,x}+\frac{\ell\ell_y}{\ell_x^2+\ell_y^2}\varphi_{2,y}\right) \right\|_{T_2}\\
		\leqslant & \|\beta_2 \nabla \varphi_2\|_{T_2}+\left\|\nabla\left(\frac{\ell\ell_x}{\ell_x^2+\ell_y^2}\right) \right\|_{\infty,T_2} \|\beta_2  \varphi_{2,x}\|_{T_2}+\left\|\frac{\ell\ell_x}{\ell_x^2+\ell_y^2} \right\|_{\infty,T_2} \|\beta_2  \nabla \varphi_{2,x}\|_{T_2}\\
		&+\left\|\nabla\left(\frac{\ell\ell_y}{\ell_x^2+\ell_y^2}\right) \right\|_{\infty,T_2} \|\beta_2  \varphi_{2,y}\|_{T_2}+\left\|\frac{\ell\ell_y}{\ell_x^2+\ell_y^2} \right\|_{\infty,T_2} \|\beta_2  \nabla \varphi_{2,y}\|_{T_2}\\
		\leqslant&C\|\beta_2 \nabla \varphi_2\|_{T_2}.
	\end{split}
\end{align}
Substituting the estimates \eqref{Trace_proof_2}-\eqref{Trace_proof_4} into equality \eqref{Trace_proof_3} yields
\begin{align}
		\|\beta_1 \nabla \varphi_1\|_{L^2(AD)} \leqslant Ch_T^{-\tfrac{1}{2}}\|\beta \nabla \varphi\|_{T}.
\end{align}
Combining with the estimate \eqref{Trace_proof_1}, we get
\begin{align}
	\|\beta \nabla \varphi\|_{L^2(AC)}\leqslant Ch_T^{-\tfrac{1}{2}}\|\beta \nabla \varphi\|_{T}.
\end{align}

For the non-interface edge $AB$, according to the trace inequality, the estimates \eqref{Trace_proof_2} and \eqref{Trace_proof_4}, we derive
\begin{align*}
	\begin{split}
		\|\beta \nabla \varphi\|_{L^2(AB)}=&\|\beta_1 \nabla \varphi_1\|_{L^2(AB)}\\
		\leqslant&Ch_T^{-\tfrac{1}{2}} \|\beta_1 \nabla \varphi_1\|_T+Ch_T^{\tfrac{1}{2}}\|\nabla(\beta_1 \nabla \varphi_1)\|_T\\
		\leqslant& Ch_T^{-\tfrac{1}{2}} \|\beta_1 \nabla \varphi_1\|_T+Ch_T^{-\tfrac{1}{2}} \|\beta \nabla \varphi\|_T\\
		\leqslant& Ch_T^{-\tfrac{1}{2}}(\|\beta_1 \nabla \varphi_1\|_{T_1}+\|\beta_1 \nabla \varphi_1\|_{T_2})+Ch_T^{-\tfrac{1}{2}} \|\beta \nabla \varphi\|_T\\
		\leqslant&Ch_T^{-\tfrac{1}{2}} \|\beta \nabla \varphi\|_T.
	\end{split}
\end{align*}
To sum up, we have
\begin{align*}
	\|\beta \nabla \varphi\|_{\partial T} \leqslant&Ch_T^{-\tfrac{1}{2}}\sqrt{\beta}\|\sqrt{\beta}\nabla \varphi\|_T \leqslant Ch_T^{-\tfrac{1}{2}}\sqrt{\beta_2} \frac{\sqrt{\beta_2}}{\sqrt{\beta_1}}\|\sqrt{\beta}\nabla \varphi\|_T\leqslant Ch_T^{-\tfrac{1}{2}}\frac{\beta_2}{\sqrt{\beta_1}}\|\sqrt{\beta}\nabla \varphi\|_T.
\end{align*}

	Second, if $|CD | \leqslant \dfrac{1}{2}|AC|$ or $|CE | \leqslant \dfrac{1}{2}|BC|$, without loss of generality, we assume $|CD | \leqslant \dfrac{1}{2}|AC|$. Now, the IFE function defined on $T$ is given by Eq.\eqref{IFE_function2}.
	Similar to the estimate \eqref{Trace_proof_5},  we get the following result for $\|\nabla \varphi\|_{L^2(AD)}$:
	\begin{align}\label{Trace_proof_6}
			\|\nabla \varphi\|_{L^2(AD)}=\|\nabla \varphi_1\|_{L^2(AD)}\leqslant C h_T^{-\tfrac{1}{2}}\| \nabla \varphi_1\|_{L^2(ABED)}\leqslant C h_T^{-\tfrac{1}{2}}\| \nabla \varphi\|_{L^2(T)}.
	\end{align}
Similar to the proof of estimate \eqref{Trace_proof_2} - \eqref{Trace_proof_4}, we have
\begin{align}
		\|\nabla(\nabla\varphi_2)\|_T^2 \leqslant& C h_T^{-2}\|\nabla\varphi\|_T^2,\label{Trace_proof_7}\\
		\|\nabla\varphi_2\|_{T_1}^2 \leqslant& C \|\nabla\varphi_1\|_{T_1}^2.\label{Trace_proof_8}
\end{align}
Combining the trace inequality with the estimate \eqref{Trace_proof_7} - \eqref{Trace_proof_8} leads to
\begin{align*}
	\begin{split}
		\|\nabla \varphi\|_{L^2(CD)}=&\| \nabla \varphi_2\|_{L^2(CD)}\leqslant\| \nabla \varphi_2\|_{L^2(AC)}\\
		\leqslant& C h_T^{-\tfrac{1}{2}}\| \nabla \varphi_2\|_T+C h_T^{\tfrac{1}{2}}\|\nabla(\nabla \varphi_2)\|_T\\
		\leqslant& C h_T^{-\tfrac{1}{2}}\| \nabla \varphi_2\|_T+C h_T^{-\tfrac{1}{2}}\| \nabla \varphi\|_T\\
		\leqslant& C h_T^{-\tfrac{1}{2}}\| \nabla \varphi_2\|_{T_1}+C h_T^{-\tfrac{1}{2}}\| \nabla \varphi_2\|_{T_2}+C h_T^{-\tfrac{1}{2}}\| \nabla \varphi\|_{T}\\
        \leqslant& C h_T^{-\tfrac{1}{2}}\| \nabla \varphi_1\|_{T_1}+C h_T^{-\tfrac{1}{2}}\| \nabla \varphi_2\|_{T_2}+C h_T^{-\tfrac{1}{2}}\| \nabla \varphi\|_{T}\\
        \leqslant& C h_T^{-\tfrac{1}{2}}\| \nabla \varphi\|_{T}.
	\end{split}
\end{align*}
Therefore, we derive
\begin{align*}
		\|\nabla \varphi\|_{L^2(AC)}\leqslant\|\nabla \varphi\|_{L^2(AD)}+\|\nabla \varphi\|_{L^2(CD)} \leqslant C h_T^{-\tfrac{1}{2}}\| \nabla \varphi\|_{T}.
\end{align*}
For the non-interface edge $AB$, according to the trace inequality, the inverse inequality and Lemma \ref{Norm_equivalence}, we obtain
\begin{align*}
		\|\beta \nabla \varphi\|_{L^2(AB)}=&\|\beta_1 \nabla \varphi_1\|_{L^2(AB)}\\
		\leqslant&Ch_T^{-\tfrac{1}{2}} \|\beta_1 \nabla \varphi_1\|_T+Ch_T^{\tfrac{1}{2}}\|\nabla(\beta_1 \nabla \varphi_1)\|_T\\
		\leqslant& Ch_T^{-\tfrac{1}{2}} \|\beta_1 \nabla \varphi_1\|_T\\
		\leqslant& Ch_T^{-\tfrac{1}{2}} \|\beta_1 \nabla \varphi_1\|_{T_1}\\
		\leqslant&Ch_T^{-\tfrac{1}{2}} \|\beta \nabla \varphi\|_T.
\end{align*}
To sum up, we have
\begin{align*}
	\|\beta \nabla \varphi\|_{\partial T}\leqslant\beta_2\|\nabla \varphi\|_{\partial T} \leqslant&C\beta_2h_T^{-\tfrac{1}{2}}\|\nabla \varphi\|_T \leqslant C h_T^{-\tfrac{1}{2}} \frac{\beta_2}{\sqrt{\beta_1}}\sqrt{\beta_1}\|\nabla \varphi\|_T\leqslant Ch_T^{-\tfrac{1}{2}}\frac{\beta_2}{\sqrt{\beta_1}}\|\sqrt{\beta}\nabla \varphi\|_T.
\end{align*}
The proof of the theorem is completed.
\end{proof}

\section{Error Estimates}
In this section, we present the error estimates of $u$ under the energy norm and $L^2$ norm. Define the following norms on the space $V_h^0$:
\begin{align*}
	\|v\|_h^2=&\sum_{T \in \mathcal{T}_h}\|\sqrt{\beta}\nabla v\|_T^2+\sum_{e \in \mathcal{E}_h}\dfrac{\gamma\sigma_0}{h}\|\jump{v}\|_e^2,\\
	\trb{v}_h^2=&\|v\|_h^2+\sum_{e \in \mathcal{E}_h}\dfrac{h}{\gamma\sigma_0}\|\{\beta \nabla v \cdot \bn_e\}\|_e^2.
\end{align*}

\begin{lemma}\label{Projection_estimate_enenrgy}
	For $u \in PH^2(\Omega)$, the following estimate holds true:
	\begin{align}
		\trb{\Pi_h u -u}_h \leqslant C \dfrac{\beta_2^2}{\sqrt{\beta_1^3}}h \|u\|_{2,\Omega}^2.
	\end{align}
\end{lemma}
\begin{proof}
	For simplicity, we take $w=\Pi_h u -u$. Based on the definition of $\trb{\cdot}_h$, we analyze the each term in $\trb{w}_h$.
    
    Using Theorem \ref{Projection_estimate_global} leads to
    \begin{align*}
    	\sum_{T \in \mathcal{T}_h}\|\sqrt{\beta}\nabla w\|_T^2 \leqslant \beta_2\sum_{T \in \mathcal{T}_h}\|\nabla w\|_T^2 \leqslant C \left(\dfrac{\beta_2^3}{\beta_1^{2}}\right) h^2 \|u\|_{2,\Omega}.
    \end{align*}
    Since $w|_T \in H^1(T)$, by the trace inequality and Theorem \ref{Projection_estimate_global}, we have
    \begin{align*}
    	\begin{split}
    			&\sum_{e \in \mathcal{E}_h} \|\jump{w}\|_e^2 \leqslant C \sum_{T \in \mathcal{T}_h} \|w|_T\|_{\partial T}^2\\ 
    			\leqslant& C \sum_{T \in \mathcal{T}_h} \left( h_T^{-1}\|w\|_{T}^2+h_T\|\nabla w\|_{T}^2\right) \leqslant C\left(\dfrac{\beta_2^2}{\beta_1^2}\right) h^3\|u\|_{2,\Omega}^2.
    	\end{split}  
    \end{align*}
    For the normal vector $\bn_e$ on edge $e \in \mathcal{E}_h$, set $\bn_e=a\bn+b{\bm{\tau}}$. The third term is rewritten as 
    \begin{align*}
    	\begin{split}
    			\sum_{e \in \mathcal{E}_h}\|\{\beta \nabla v \cdot \bn_e\}\|_e^2\leqslant	C \sum_{T \in \mathcal{T}_h}\|\beta \nabla v \cdot \bn_e\|_{\partial T}^2
    		\leqslant C\sum_{T \in \mathcal{T}_h}\left(\|\beta \nabla v \cdot \bn\|_{\partial T}^2+\|\beta \nabla v \cdot {\bm{\tau}}\|_{\partial T}^2\right).
    	\end{split}
    \end{align*}
 Therefore, it follows from the trace inequality that 
 \begin{align*}
 	\begin{split}
 		&\sum_{e \in \mathcal{E}_h}\|\{\beta \nabla w \cdot \bn_e\}\|_e^2\\
 		\leqslant&C\left(\sum_{T \in \mathcal{T}_h}h_T^{-1}\|\beta \nabla w \cdot \bn\|_T^2+h_T\|\nabla(\beta \nabla w \cdot \bn)\|_T^2\right)\\
 		&+C\beta_2^2 \left(\sum_{T \in \mathcal{T}_h}h_T^{-1}\| \nabla w \cdot {\bm{\tau}}\|_T^2+h_T\|\nabla( \nabla w \cdot {\bm{\tau}})\|_T^2\right)\\
 		\leqslant& C \beta_2^2 \sum_{T \in \mathcal{T}_h}h_T^{-1} \left( \|\nabla w
 		\|_T^2+h_T^2\|\nabla(\nabla w)
 		\|_T^2 \right)\\
 		\leqslant& C\left(\dfrac{\beta_2^4}{\beta_1^2}\right) h \|u\|_{2,\Omega}^2.
 	\end{split}
 \end{align*}  
 To sum up, we get
 \begin{align*}
 	\trb{u-\Pi_h u}_h^2 \leqslant& C \left(\dfrac{\beta_2^3}{\beta_1^2}\right) h^2 \|u\|_{2,\Omega}^2+ C \dfrac{\sigma_0}{h} \left(\dfrac{\beta_2^4}{\beta_1^3}\right)h^3 \|u\|_{2,\Omega}^2
 	+Ch\dfrac{h}{\sigma_0}\left(\dfrac{\beta_2^2}{\beta_1}\right) h \|u\|_{2,\Omega}^2\\
 	\leqslant& C \left(\dfrac{\beta_2^4}{\beta_1^3}\right)h^2 \|u\|_{2,\Omega}^2.
 \end{align*} 
Thus, the proof of the lemma is completed.
\end{proof}

\begin{lemma}\label{Continuity_bilinear_form}
For the bilinear form $a_h(u_h,v_h)$,  we have 
	\begin{align}
		a_h(u_h,v_h)\leqslant \trb{u_h}_h \cdot \trb{v_h}_h.
	\end{align}
\end{lemma}

\begin{lemma}\label{Coercity_bilinear}
	For any sufficiently large constant $\sigma_0 >0$, the following estimate holds true
	\begin{align}
		a_h(v_h,v_h)\geqslant \frac{1}{4}  \trb{v_h}_h, \quad \forall \, v_h \in V_h^0.
	\end{align}
\end{lemma}
\begin{proof}
	By the definition of the bilinear form $a_h(v_h,v_h)$, we have
	\begin{align}\label{Coercity_bilinear_proof_1}
		\begin{split}
			a_h(v_h,v_h)=\sum_{T \in \mathcal{T}_h}\|\sqrt{\beta}\nabla v_h\|_T^2-2\sum_{e \in \mathcal{E}_h}\langle \{\beta \nabla v_h \cdot \bn_e\}_e, \jump{v_h}_e \rangle_{e}
			+\sum_{e \in \mathcal{E}_h}\dfrac{\gamma\sigma_0}{h}\|\jump{v_h}_e\|_e^2.
		\end{split}
	\end{align}
	For the second term, using the Cauchy-Schwarz inequality and Young inequality leads to
	\begin{align}\label{Coercity_bilinear_proof_2}
		\begin{split}
			&\sum_{e \in \mathcal{E}_h}\langle \{\beta \nabla v_h \cdot \bn_e\}_e, \jump{v_h}_e \rangle_{e}
			\leqslant \sum_{e \in \mathcal{E}_h} \|\{\beta \nabla v_h \cdot \bn_e\}_e\|_e \|\jump{v_h}_e\|_e\\
			\leqslant&\left( \sum_{e \in \mathcal{E}_h} h\|\{\beta \nabla v_h \cdot \bn_e\}_e\|_e^2 \right)^{\tfrac{1}{2}} \left( \sum_{e \in \mathcal{E}_h} \dfrac{1}{h}\|\jump{v_h}_e\|_e^2 \right)^{\tfrac{1}{2}}\\
			\leqslant&\sum_{e \in \mathcal{E}_h} \dfrac{h}{2 \alpha \gamma}  \|\{\beta \nabla v_h \cdot \bn_e\}_e\|_e^2 + \sum_{e \in \mathcal{E}_h} \dfrac{\alpha \gamma}{2h} \|\jump{v_h}_e\|_e^2.
		\end{split}
	\end{align}
	As to the first term in the above inequality, it follows from the estimate \eqref{Trace_inverse_inequality_interface} that 
	\begin{align}\label{Coercity_bilinear_proof_3}
		\begin{split}
			\dfrac{1}{2 \alpha \gamma}\sum_{e \in \mathcal{E}_h}h\|\{\beta \nabla v_h \cdot \bn_e\}_e\|_e^2\leqslant&  \dfrac{1}{2 \alpha \gamma}\sum_{T \in \mathcal{T}_h}h\|\beta \nabla v_h \cdot \bn_e\|_{\partial T}^2\\
			\leqslant& \dfrac{1}{2 \alpha \gamma}\sum_{T \in \mathcal{T}_h}h\|\beta \nabla v_h\|_{\partial T}^2\\
			\leqslant& \dfrac{C^2}{2 \alpha \gamma} \sum_{T \in \mathcal{T}_h} \dfrac{\beta_2^2}{\beta_1}\|\sqrt{\beta} \nabla v_h\|_{ T}^2 \\
			\leqslant & \dfrac{C^2}{2 \alpha} \sum_{T \in \mathcal{T}_h} \|\sqrt{\beta} \nabla v_h\|_{T}^2.
		\end{split}
	\end{align}
	Substituting inequalities \eqref{Coercity_bilinear_proof_2}-\eqref{Coercity_bilinear_proof_3} into \eqref{Coercity_bilinear_proof_1} yields 
	\begin{align}
		a_h(v_h,v_h)\geqslant \sum_{T \in \mathcal{T}_h}\left(1-\dfrac{C^2}{\alpha}\right) \|\sqrt{\beta} \nabla v_h\|_{T}^2 +\sum_{e \in \mathcal{E}_h} \left( \dfrac{\gamma(\sigma_0-\alpha)}{h}\right) \|\jump{v_h}_e\|_e^2.
	\end{align}
	In particular, taking $\alpha=\dfrac{1}{2}\sigma_0$ and  $\sigma_0 \geqslant 4C^2+1$ leads to
	\begin{align}
			a_h(v_h,v_h)\geqslant \dfrac{1}{2} \sum_{T \in \mathcal{T}_h} \|\sqrt{\beta} \nabla v_h\|_{T}^2 +\dfrac{1}{2} \sum_{e \in \mathcal{E}_h} \dfrac{\gamma \sigma_0}{h} \|\jump{v_h}_e\|_e^2.
	\end{align}
	According to the inequality \eqref{Coercity_bilinear_proof_3}, we have
	\begin{align}
		\sum_{e \in \mathcal{E}_h}\dfrac{h}{\sigma_0 \gamma} \|\{\beta \nabla v_h \cdot \bn_e\}_e\|_e^2 \leqslant \dfrac{C^2}{\sigma_0} \sum_{T \in \mathcal{T}_h} \|\sqrt{\beta}\nabla v_h\|_T^2\leqslant \dfrac{1}{4} \sum_{T \in \mathcal{T}_h} \|\sqrt{\beta}\nabla v_h\|_T^2.
	\end{align}
Therefore, we obtain
\begin{align}
	\begin{split}
			a_h(v_h,v_h)=&\sum_{T \in \mathcal{T}_h}\|\sqrt{\beta}\nabla v_h\|_T^2-2\sum_{e \in \mathcal{E}_h}\langle \{\beta \nabla v_h \cdot \bn_e\}_e, \jump{v_h}_e \rangle_{e}
		+\sum_{e \in \mathcal{E}_h}\dfrac{\gamma\sigma_0}{h}\|\jump{v_h}_e\|_e^2\\
		\geqslant &\dfrac{1}{2} \sum_{T \in \mathcal{T}_h} \|\sqrt{\beta} \nabla v_h\|_{T}^2 +\dfrac{1}{2} \sum_{e \in \mathcal{E}_h} \dfrac{\gamma \sigma_0}{h} \|\jump{v_h}_e\|_e^2\\
		\geqslant & \dfrac{1}{4} \sum_{T \in \mathcal{T}_h} \|\sqrt{\beta} \nabla v_h\|_{T}^2+\dfrac{1}{2} \sum_{e \in \mathcal{E}_h} \dfrac{\gamma \sigma_0}{h} \|\jump{v_h}_e\|_e^2+\sum_{e \in \mathcal{E}_h}\dfrac{h}{\sigma_0 \gamma} \|\{\beta \nabla v_h \cdot \bn_e\}_e\|_e^2\\
		\geqslant & \dfrac{1}{4} \trb{v_h}_h^2.
	\end{split}
\end{align}
The proof of the above lemma is completed.
\end{proof}

\begin{theorem}
	Assume $u \in PH^2(\Omega)$ is the solution of the interface problem \eqref{Elliptic_problem} - \eqref{Interface_condition_2} and the $\sigma_0$ is large enough such that Lemma \ref{Coercity_bilinear} to hold, then we have the following result
	\begin{align}
		\trb{u-u_h}_h \leqslant C \left(\dfrac{\beta_2^2}{\sqrt{\beta_1^3}}\right)h\|u\|_{2,\Omega}. 
	\end{align}
	 and 
	 \begin{align}
	 	\|u-u_h\|_{L^2(\Omega)}\leqslant C \left( \dfrac{\beta_2}{\beta_1} \right)^4 h^2 \|u\|_{2,\Omega}.
	 \end{align}
\end{theorem}

\section{The Numerical Experiments}
In this section, several numerical examples are presented to verify the theoretical results established in the previous sections. Furthermore, we examine the numerical performance of the proposed immersed SIPDG method for various computational examples.

\begin{example}\label{li1}{\rm({\textbf{Convergence verification}})}
	Consider the interface problems in the square domain $[-1,1]\times[-1,1]$. The interface function is as follows:
	$$
	x^2+y^2=\frac{1}{3}.
	$$
	Take two subdomains as $\Omega_1=\left\{(x,y)\in \Omega: x^2+y^2<\dfrac{1}{3}\right\}$ and  $\Omega_2=\left\{(x,y)\in \Omega: x^2+y^2>\dfrac{1}{3} \right\}$. The exact solution is
	$$
	u = \left\{\begin{array}{lc}
		\dfrac{1}{\beta_1}\cos(\pi (x^2+y^2)), & (x,y)\in \Omega_1 \\
		\dfrac{1}{\beta_2}\cos(\pi (x^2+y^2))+\dfrac{1}{2}(\dfrac{1}{\beta_1}-\dfrac{1}{\beta_2}), & (x,y)\in \Omega_2
	\end{array}
	\right..
	$$
\end{example}

\begin{figure}[H]
	\centering
	\begin{minipage}[t]{0.33\linewidth}
		\centering
		\includegraphics[width=1.1\linewidth]{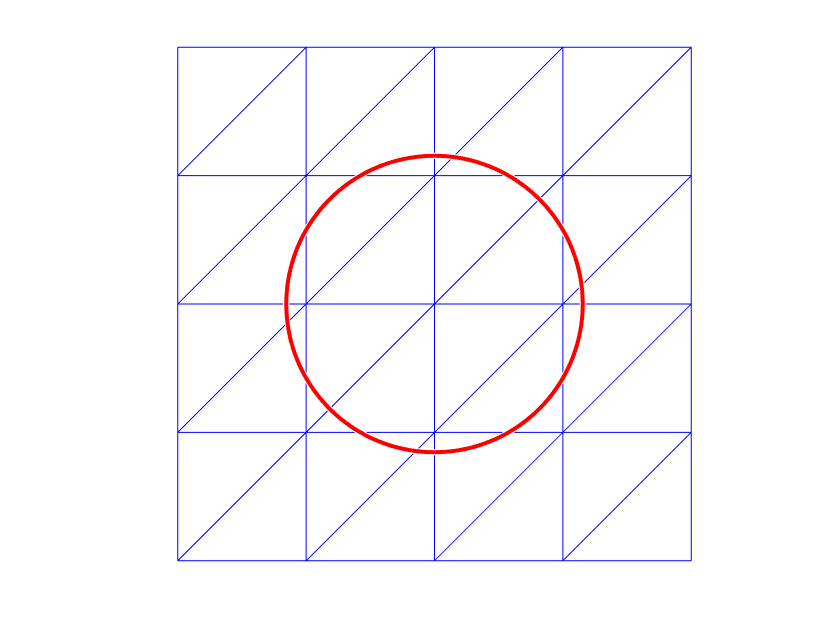}
	\end{minipage}%
	\begin{minipage}[t]{0.33\linewidth}
		\centering
		\includegraphics[width=1.1\linewidth]{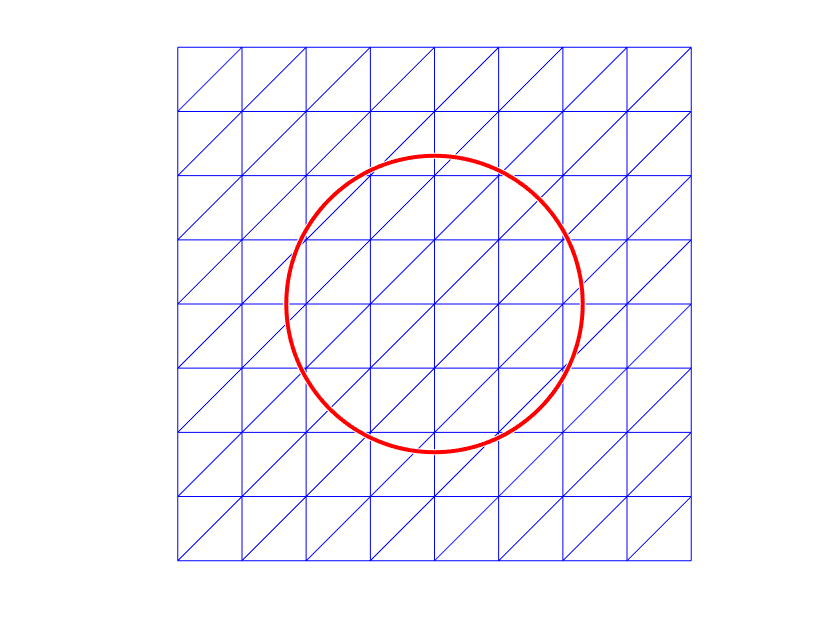}
	\end{minipage}
	\begin{minipage}[t]{0.33\linewidth}
		\centering \includegraphics[width=1.1\linewidth]{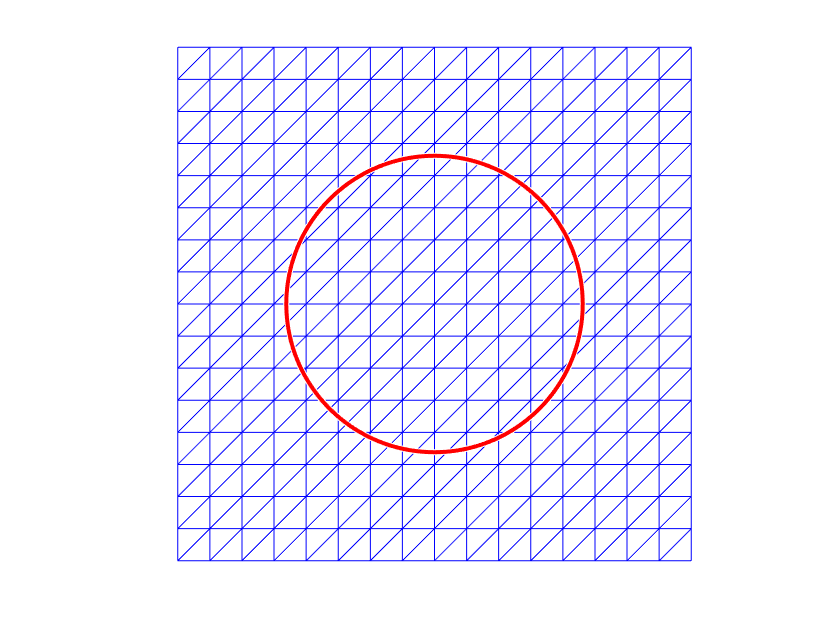}
	\end{minipage}
	\caption{  Example \ref{li1}: The meshes of domain $\Omega$ with $n$ = 4, 8, 16.}
	\label{partition_2_li1}
	\vspace*{-0.5cm}
\end{figure}

\begin{table}[H]
	\centering
	\caption{Example \ref{li1}: Projection errors and convergence rates on triangular meshes with $\beta_2 \geqslant \beta_1$.}
	\centering
	\begin{tabular}{ccccc|cccc}
		\hline
		n&$\|\Pi_h u -u \|_{1}$&order&$\|\Pi_h u -u\|$&order&$\|\Pi_h u -u \|_{1}$&order&$\|\Pi_h u -u\|$&order\\
		\hline		
		\multicolumn{5}{c|}{$\beta_1=1,\beta_2=1$}&\multicolumn{4}{c}{$\beta_1=1,\beta_2=10$}\\
		\hline
        32    & 7.3113E-01 &--  & 5.6412E-03 & --  & 1.8687E-01 &--  & 3.9526E-03 & -- \\
	    64    & 3.6655E-01 & 0.9961  & 1.4147E-03 & 1.9955  & 9.4919E-02 & 0.9772  & 1.0052E-03 & 1.9753  \\
    	128   & 1.8340E-01 & 0.9990  & 3.5394E-04 & 1.9989  & 4.7898E-02 & 0.9867  & 2.5430E-04 & 1.9829  \\
	    256   & 9.1715E-02 & 0.9998  & 8.8503E-05 & 1.9997  & 2.4054E-02 & 0.9937  & 6.3752E-05 & 1.9960  \\
	    512   & 4.5859E-02 & 0.9999  & 2.2127E-05 & 1.9999  & 1.2054E-02 & 0.9967  & 1.5964E-05 & 1.9977  \\
		\hline
		\multicolumn{5}{c|}{$\beta_1=1,\beta_2=100$}&\multicolumn{4}{c}{$\beta_1=1,\beta_2=1000$}\\
		\hline
        32    & 1.4335E-01 & --  & 1.1656E-03 & -- & 1.4328E-01 & -- & 1.1639E-03 & --  \\
    	64    & 7.2772E-02 & 0.9780  & 2.9662E-04 & 1.9744  & 7.2727E-02 & 0.9782  & 2.9615E-04 & 1.9746  \\
	    128   & 3.6759E-02 & 0.9853  & 7.5259E-05 & 1.9787  & 3.6718E-02 & 0.9860  & 7.5166E-05 & 1.9782  \\
    	256   & 1.8457E-02 & 0.9939  & 1.8916E-05 & 1.9923  & 1.8436E-02 & 0.9940  & 1.8893E-05 & 1.9922  \\
	    512   & 9.2497E-03 & 0.9967  & 4.7430E-06 & 1.9957  & 9.2387E-03 & 0.9968  & 4.7378E-06 & 1.9956  \\	
		\hline
	\end{tabular}
	\label{table1}
	\vspace*{-0.5cm}
\end{table}

\begin{table}[H]
	\centering
	\caption{Example \ref{li1}: Projection errors and convergence rates on triangular meshes with $\beta_2 \leqslant \beta_1$.}
	\centering
	\begin{tabular}{ccccc|cccc}
		\hline
		n&$\|\Pi_h u -u \|_{1}$&order&$\|\Pi_h u -u\|$&order&$\|\Pi_h u -u \|_{1}$&order&$\|\Pi_h u -u\|$&order\\
		\hline		
		\multicolumn{5}{c|}{$\beta_1=1,\beta_2=1$}&\multicolumn{4}{c}{$\beta_1=10,\beta_2=1$}\\
		\hline
		4     & 4.6119E+00 & --  & 2.7761E-01 & --  & 4.5010E+00 & --  & 2.7008E-01 & --  \\
		8     & 2.7654E+00 & 0.7379  & 8.4424E-02 & 1.7173  & 2.6983E+00 & 0.7382  & 8.1969E-02 & 1.7202  \\
		16    & 1.4465E+00 & 0.9349  & 2.2283E-02 & 1.9217  & 1.4116E+00 & 0.9347  & 2.1637E-02 & 1.9216  \\
		32    & 7.3113E-01 & 0.9844  & 5.6412E-03 & 1.9819  & 7.1533E-01 & 0.9806  & 5.5010E-03 & 1.9757  \\
		64    & 3.6655E-01 & 0.9961  & 1.4147E-03 & 1.9955  & 3.5883E-01 & 0.9953  & 1.3805E-03 & 1.9945  \\
		
		\hline
		\multicolumn{5}{c|}{$\beta_1=100,\beta_2=1$}&\multicolumn{4}{c}{$\beta_1=1000,\beta_2=1$}\\
		\hline
		4     & 4.4911E+00 & --  & 2.6950E-01 & --  & 4.4913E+00 & --  & 2.6948E-01 & --  \\
		8     & 2.6985E+00 & 0.7349  & 8.1880E-02 & 1.7187  & 2.6988E+00 & 0.7348  & 8.1875E-02 & 1.7187  \\
		16    & 1.4110E+00 & 0.9355  & 2.1616E-02 & 1.9214  & 1.4110E+00 & 0.9356  & 2.1614E-02 & 1.9214  \\
		32    & 7.1511E-01 & 0.9804  & 5.4984E-03 & 1.9750  & 7.1512E-01 & 0.9805  & 5.4982E-03 & 1.9749  \\
		64    & 3.5876E-01 & 0.9951  & 1.3801E-03 & 1.9943  & 3.5877E-01 & 0.9951  & 1.3801E-03 & 1.9942  \\
		
		\hline
	\end{tabular}
	\label{table9}
	\vspace*{-0.5cm}
\end{table}

We solve the interface problems on uniform triangular meshes (see Figure \ref{partition_2_li1}) and consider two cases corresponding to the coefficients $\beta_1 \geqslant \beta_2$ and  $\beta_2 \geqslant \beta_1$, respectively. The projection errors and numerical errors under different coefficients are reported in Tables \ref{table1} - \ref{table10}. From these numerical results, it can be observed that both the projection function $\Pi_h u$ and the numerical solution $u_h$ converge optimally to the exact solution $u$. Moreover, the convergence rates are of order $O(h^1)$ in the $H^1$ norm and $O(h^{2})$ in the $L^2$ norm, respectively. Figure \ref{Figure1_IEI} presents the condition numbers of the stiffness matrices under different coefficients. The numerical results show that the condition numbers grow approximately at the rate of $h^{-2}$.

\begin{table}[H]
	\centering
	\caption{Example \ref{li1} : Numerical errors and convergence rates on triangular meshes with $\beta_2 \geqslant \beta_1$.}
	\centering
	\begin{tabular}{ccccc|cccc}
		\hline
		n&$\3bar u -u_h \3bar_h$&order&$\| u -u_h\|$&order&$\3bar u -u_h \3bar_h$&order&$\| u -u_h\|$&order\\
		\hline		
		\multicolumn{5}{c|}{$\beta_1=1,\beta_2=1$}&\multicolumn{4}{c}{$\beta_1=1,\beta_2=10$}\\
		\hline
        32    & 8.0595E-01 & --  & 1.4526E-02 &--  & 1.6048E-01 & --  & 1.2938E-03 & --  \\
        64    & 4.0135E-01 & 1.0058  & 3.8333E-03 & 1.9220  & 8.1211E-02 & 0.9827  & 3.2776E-04 & 1.9809  \\
        128   & 1.9997E-01 & 1.0051  & 9.8170E-04 & 1.9652  & 4.0893E-02 & 0.9898  & 8.2819E-05 & 1.9846  \\
        256   & 9.9775E-02 & 1.0030  & 2.4818E-04 & 1.9839  & 2.0514E-02 & 0.9953  & 2.0791E-05 & 1.9940  \\
        512   & 4.9832E-02 & 1.0016  & 6.2377E-05 & 1.9923  & 1.0274E-02 & 0.9976  & 5.2090E-06 & 1.9968  \\
		\hline
		\multicolumn{5}{c|}{$\beta_1=1,\beta_2=100$}&\multicolumn{4}{c}{$\beta_1=1,\beta_2=1000$}\\
		\hline
        32    & 1.7893E-01 & -- & 5.8733E-03 & --  & 3.0108E-01 & --  &  2.4932E-02 & -- \\
    	64    & 8.8981E-02 & 1.0079  & 1.1728E-03 & 2.3242  & 1.1214E-01 & 1.4248  & 4.2561E-03 & 2.5504  \\
    	128   & 4.4687E-02 & 0.9937  & 2.7699E-04 & 2.0821  & 4.7289E-02 & 1.2457  & 8.0676E-04 & 2.3993  \\
	    256   & 2.2394E-02 & 0.9967  & 6.3203E-05 & 2.1318  & 2.2830E-02 & 1.0506  & 1.2078E-04 & 2.7398  \\
    	512   & 1.1218E-02 & 0.9973  & 1.5306E-05 & 2.0459  & 1.1310E-02 & 1.0133  & 2.2154E-05 & 2.4467  \\
		\hline
	\end{tabular}
	\label{table2}
	\vspace*{-0.5cm}
\end{table}

\begin{table}[H]
	\centering
	\caption{Example \ref{li1}: Numerical errors and convergence rates on triangular meshes with $\beta_2 \leqslant \beta_1$.}
	\centering
	\begin{tabular}{ccccc|cccc}
		\hline
		n&$\3bar u -u_h \3bar_h$&order&$\| u -u_h\|$&order&$\3bar u -u_h \3bar_h$&order&$\| u -u_h\|$&order\\
		\hline		
		\multicolumn{5}{c|}{$\beta_1=1,\beta_2=1$}&\multicolumn{4}{c}{$\beta_1=10,\beta_2=1$}\\
		\hline
	    4     & 4.9544E+00 & --  & 4.4953E-01 & --  & 5.2734E+00 & --  & 7.2758E-01 & --  \\
	8     & 3.0411E+00 & 0.7041  & 1.6191E-01 & 1.4733  & 3.3002E+00 & 0.6762  & 2.8049E-01 & 1.3752  \\
	16    & 1.6041E+00 & 0.9228  & 5.1470E-02 & 1.6533  & 1.7757E+00 & 0.8941  & 7.7018E-02 & 1.8647  \\
	32    & 8.0595E-01 & 0.9930  & 1.4526E-02 & 1.8251  & 9.0834E-01 & 0.9671  & 2.0201E-02 & 1.9308  \\
	64    & 4.0135E-01 & 1.0058  & 3.8333E-03 & 1.9220  & 4.5690E-01 & 0.9914  & 5.1229E-03 & 1.9794  \\
		\hline
		\multicolumn{5}{c|}{$\beta_1=100,\beta_2=1$}&\multicolumn{4}{c}{$\beta_1=1000,\beta_2=1$}\\
		\hline
	    4     & 5.5325E+00 & --  & 9.1694E-01 & --  & 6.0594E+00 & --  & 1.5111E+00 & --  \\
	8     & 3.5968E+00 & 0.6212  & 4.7295E-01 & 0.9551  & 4.4193E+00 & 0.4554  & 8.9466E-01 & 0.7562  \\
	16    & 1.8309E+00 & 0.9742  & 9.2004E-02 & 2.3619  & 2.0099E+00 & 1.1367  & 1.8175E-01 & 2.2994  \\
	32    & 9.3279E-01 & 0.9729  & 2.2754E-02 & 2.0155  & 9.7122E-01 & 1.0493  & 4.0300E-02 & 2.1731  \\
	64    & 4.6839E-01 & 0.9938  & 5.5939E-03 & 2.0242  & 4.7757E-01 & 1.0241  & 9.5550E-03 & 2.0764  \\	
		\hline
	\end{tabular}
	\label{table10}
	\vspace*{-0.3cm}
\end{table}

\begin{figure}[H]
	\centering
	\begin{minipage}[t]{0.48\linewidth}
		\centering
		\includegraphics[width=1.0\linewidth]{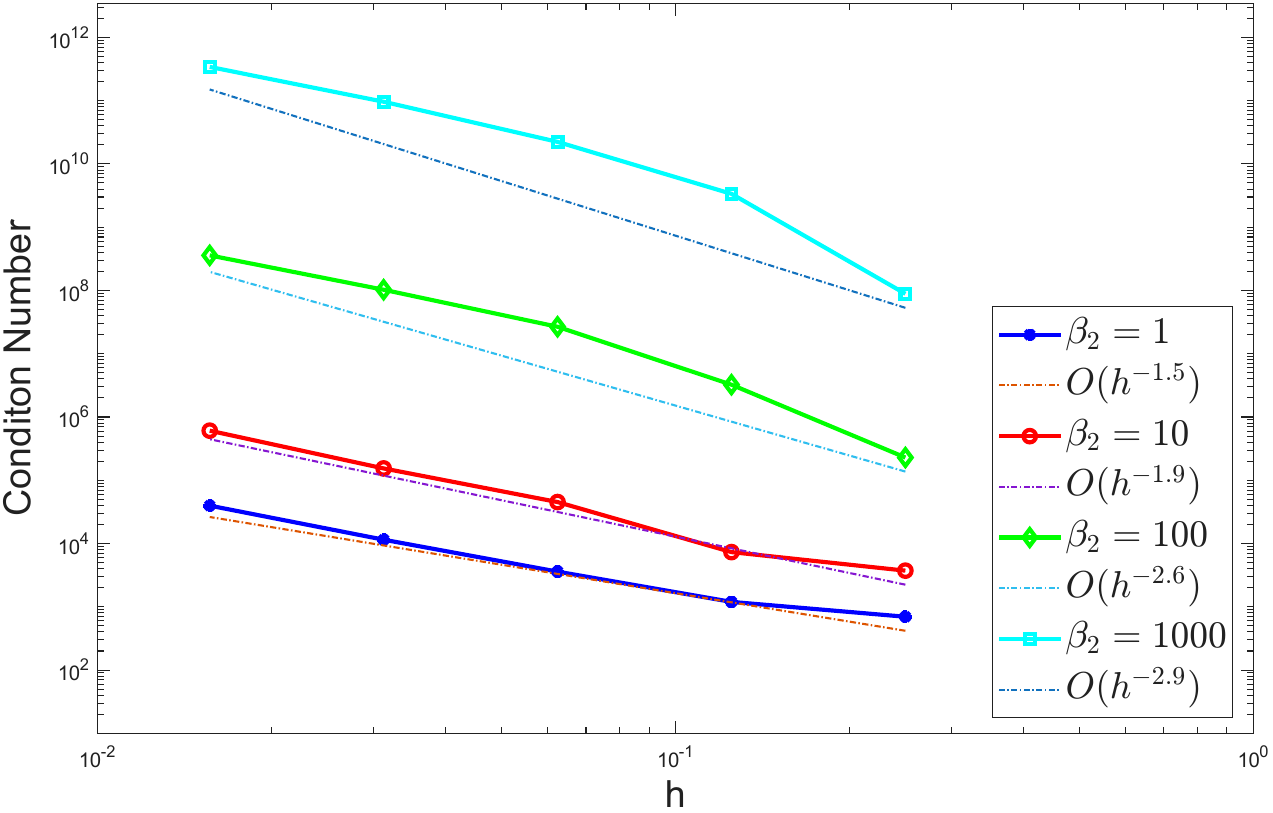}
	\end{minipage}
	\begin{minipage}[t]{0.48\linewidth}
		\centering
		\includegraphics[width=1.0\linewidth]{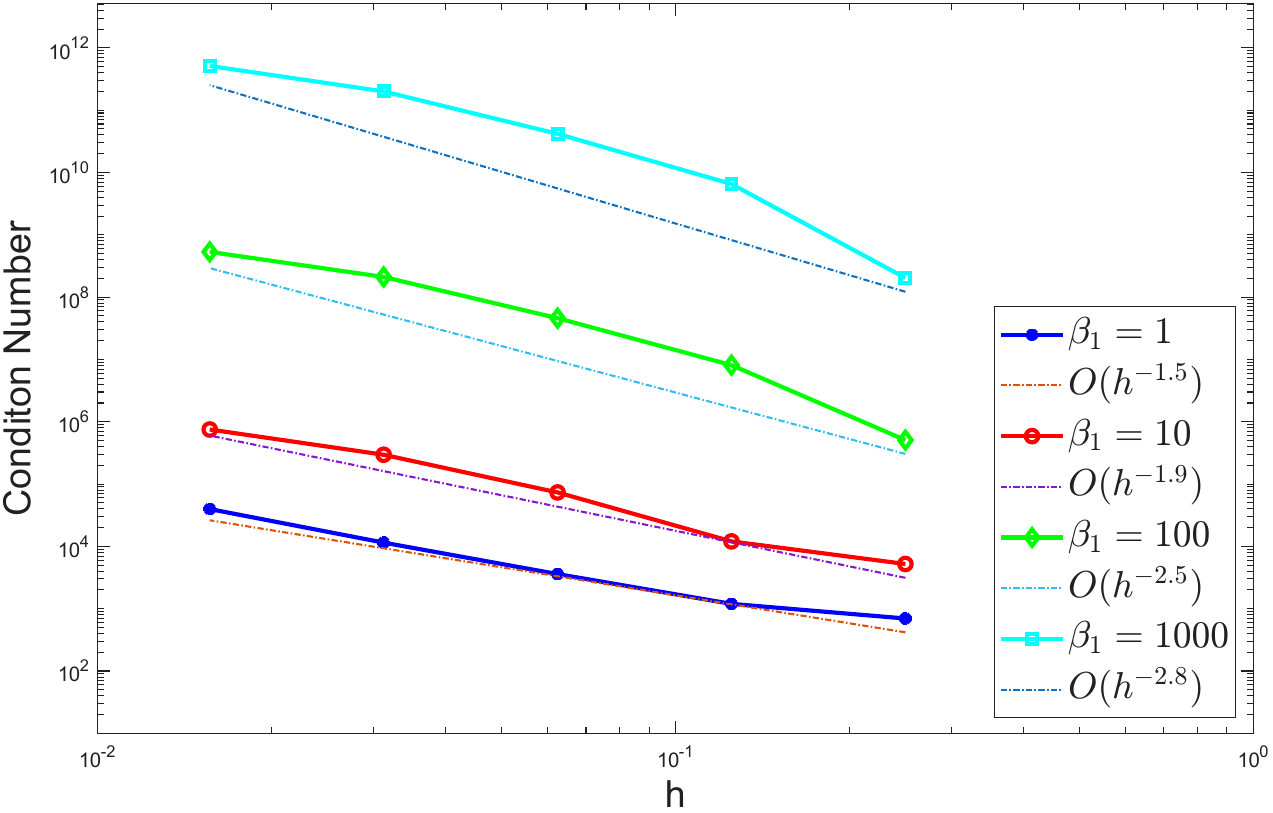}
	\end{minipage}
	\caption{Example \ref{li1}: The Condition numbers of the stiffness matrices. (Left: $\beta_1=1$,  Right: $\beta_2=1$).}
	\label{Figure1_IEI}
	\vspace*{-0.5cm}
\end{figure}

\begin{example}\label{li3}{\rm({\textbf{Accuracy verification}})}
	We consider the standard benchmark problem introduced in \cite{IFEM2010}. Let $\Omega=[-1,1]\times[-1,1]$ be the computational domain. The interface is defined by the following function:
	$$
	x^2+y^2=\frac{1}{4}.
	$$
	Take two subdomains: $\Omega_1=\Big\{(x,y)\in \Omega: x^2+y^2<\dfrac{1}{4}\Big\}$ and  $\Omega_2=\Big\{(x,y)\in \Omega: x^2+y^2>\dfrac{1}{4} \Big\}$. The exact solution is
	$$
	u = \left\{\begin{array}{lc}
		\dfrac{1}{\beta_1}(x^2+y^2)^{\tfrac{3}{2}}, & (x,y)\in \Omega_1, \\
		\dfrac{1}{\beta_2}(x^2+y^2)^{\tfrac{3}{2}}+\dfrac{1}{8}(\dfrac{1}{\beta_1}-\dfrac{1}{\beta_2}), & (x,y)\in \Omega_2.
	\end{array}
	\right.
	$$
\end{example}

	In \cite{IFEM2010}, the authors constructed piecewise linear IFE functions on interface elements by taking the edge averages as the degrees of freedom. For this example, the numerical results for the cases $\beta_1=1,\beta_2=1000$ and $\beta_1=1000,\beta_2=1$ are presented in Table \ref{table11}. By comparing the results in Table \ref{table11} with those reported in Tables 6.1 and 6.2 of \cite{IFEM2010}, we observe that numerical accuracy is comparable to that of the standard IFE methods under the same partition.

\begin{table}[H]
	\centering
	\caption{Example \ref{li3}: Projection errors and numerical errors on triangular meshes.}
	\centering
	\begin{tabular}{ccccccccc}
		\hline
		n&$\|\Pi_h u -u \|_{1}$&order&$\|\Pi_h u -u\|$&order&$\3bar u -u_h \3bar_h$&order&$\| u -u_h\|$&order\\
		\hline		
		\multicolumn{9}{c}{$\beta_1=1,\beta_2=1000$}\\
		\hline
        4     & 3.7005E-01 &--  & 7.0739E-02 & --  & 1.9451E-01 &--  & 1.1985E-02 & --  \\
        8     & 3.0933E-01 & 0.2586  & 5.4092E-02 & 0.3871  & 1.1829E-01 & 0.7175  & 3.7987E-03 & 1.6577  \\
        16    & 1.4705E-01 & 1.0728  & 2.0607E-02 & 1.3923  & 6.0072E-02 & 0.9776  & 9.8151E-04 & 1.9524  \\
        32    & 7.2122E-02 & 1.0278  & 7.0048E-03 & 1.5567  & 3.1182E-02 & 0.9460  & 2.5682E-04 & 1.9343  \\
        64    & 2.5285E-02 & 1.5122  & 1.1725E-03 & 2.5787  & 1.5873E-02 & 0.9741  &  6.5672E-05 & 1.9674  \\
        128   & 1.0834E-02 & 1.2228  & 1.8703E-04 & 2.6483  & 8.0157E-03 & 0.9857  & 1.6625E-05 & 1.9820  \\ 
		\hline
	    \multicolumn{9}{c}{$\beta_1=1000,\beta_2=1$}\\
		\hline
		4     & 1.7003E+00 & --  & 1.9483E-01 & --  & 1.3045E+00 & --  & 8.6452E-02 & --  \\
		8     & 9.4047E-01 & 0.8543  & 7.4456E-02 & 1.3877  & 6.5707E-01 & 0.9894  & 2.1841E-02 & 1.9849  \\
		16    & 4.3930E-01 & 1.0982  & 1.9498E-02 & 1.9330  & 3.3028E-01 & 0.9924  & 5.4980E-03 & 1.9900  \\
		32    & 2.1346E-01 & 1.0413  & 5.8087E-03 & 1.7471  & 1.6548E-01 & 0.9970  & 1.3784E-03 & 1.9959  \\
		64    & 1.0237E-01 & 1.0602  & 1.2412E-03 & 2.2265  & 8.2801E-02 & 0.9989  & 3.4496E-04 & 1.9985  \\
		128   & 5.0202E-02 & 1.0279  & 2.1134E-04 & 2.5541  & 4.1416E-02 & 0.9994  & 8.6286E-05 & 1.9992  \\	
		\hline
	\end{tabular}
	\label{table11}
\end{table}

\begin{example}\label{li2}{\rm({{$\mathbf{\jump{\beta}\nabla u \cdot {\bm{\tau}} \neq 0}$}})}
	Consider the interface problem discussed in \cite{IFEHigh_2024},  where the square domain $\Omega=[0.6,1.6]\times[0.21,1.21]$. The interface function is as follows:
	$$
	\Gamma=\left\{
	(x,y)\in \Omega|L(x,y)=(x^2-y^2)^2-4x^2y^2+\dfrac{1}{2}=0 \right\}.
	$$
	Let $\Omega_1=\left\{(x,y)\in \Omega| L(x,y)>0 \right\}$ and  $\Omega_2=\left\{(x,y)\in \Omega| L(x,y)<0 \right\}$. Next, let $\widetilde{L}$ denote the harmonic conjugate of $L$ given by 
	$$
	\widetilde{L}=4xy(x^2-y^2).
	$$
	We choose $f$ and $g$ such that 
	$$
	u(x,y)=\frac{\beta}{L(x,y)}+\widetilde{L}(x,y)+\frac{\beta}{L(x,y)\widetilde{L}(x,y)}.
	$$
\end{example}

\begin{figure}[H]
	\centering
	\begin{minipage}[t]{0.33\linewidth}
		\centering
		\includegraphics[width=1.1\linewidth]{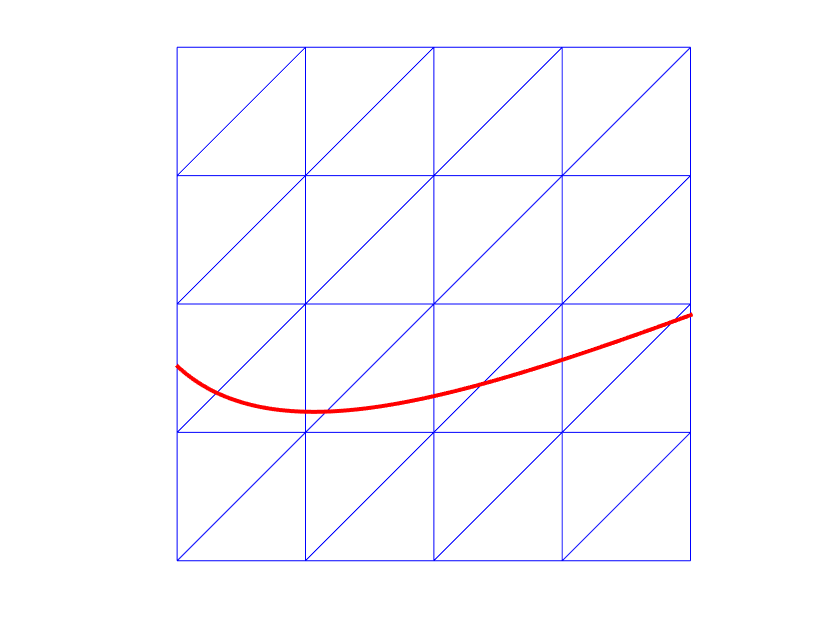}
	\end{minipage}%
	\begin{minipage}[t]{0.33\linewidth}
		\centering
		\includegraphics[width=1.1\linewidth]{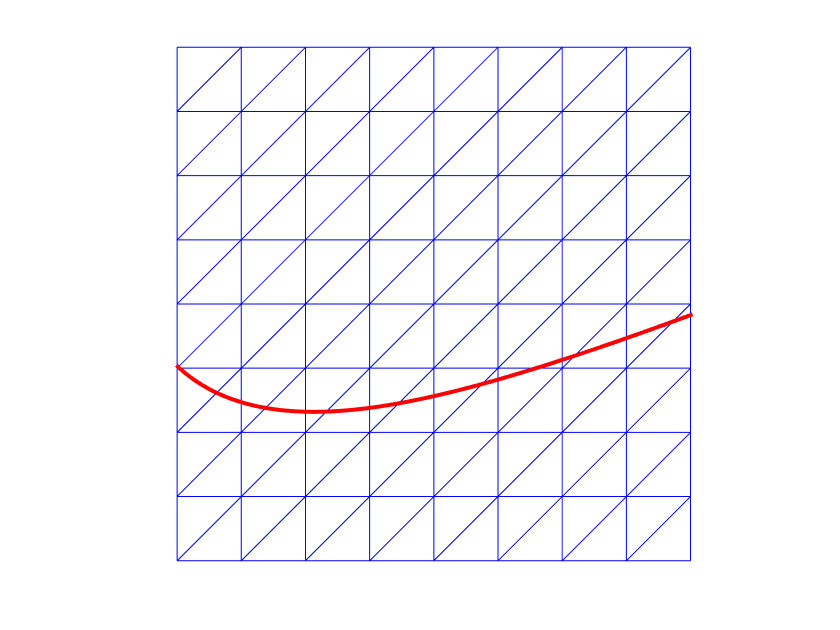}
	\end{minipage}
	\begin{minipage}[t]{0.33\linewidth}
		\centering \includegraphics[width=1.1\linewidth]{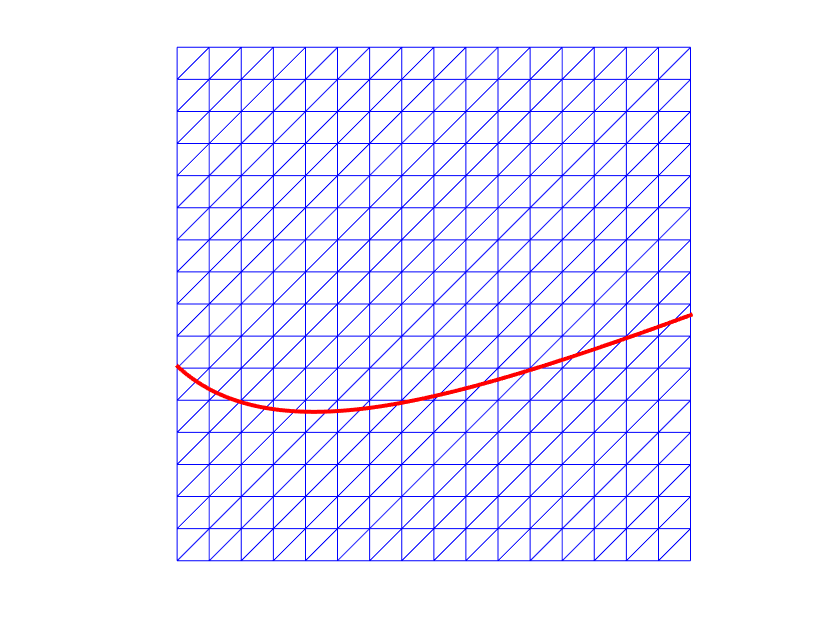}
	\end{minipage}
	\caption{  Example \ref{li2}: The meshes of domain $\Omega$ with $n$ = 4, 8, 16.}
	\label{partition_2_li2}
\end{figure}

\begin{table}[H]
	\centering
	\caption{Example \ref{li2}: Projection errors and convergence rates on triangular meshes.}
	\centering
	\begin{tabular}{ccccc|cccc}
		\hline
		n&$\|\Pi_h u -u \|_{1}$&order&$\|\Pi_h u -u\|$&order&$\|\Pi_h u -u \|_{1}$&order&$\|\Pi_h u -u\|$&order\\
		\hline		
		\multicolumn{5}{c|}{$\beta_1=1,\beta_2=1$}&\multicolumn{4}{c}{$\beta_1=1,\beta_2=10$}\\
		\hline
32    & 4.7244E+00 &--  & 1.6039E-02 & --  & 4.3881E+00 & --  & 1.5057E-02 & -- \\
64    & 2.3631E+00 & 0.9994  & 4.0113E-03 & 1.9994  & 2.1948E+00 & 0.9995  & 3.7663E-03 & 1.9992  \\
128   & 1.1817E+00 & 0.9999  & 1.0029E-03 & 1.9998  & 1.0976E+00 & 0.9998  & 9.4183E-04 & 1.9996  \\
256   & 5.9085E-01 & 1.0000  & 2.5074E-04 & 2.0000  & 5.4882E-01 & 0.9999  & 2.3548E-04 & 1.9998  \\
512   & 2.9543E-01 & 1.0000  & 6.2686E-05 & 2.0000  & 2.7442E-01 & 1.0000  & 5.8874E-05 & 1.9999  \\

		\hline
		\multicolumn{5}{c|}{$\beta_1=1,\beta_2=100$}&\multicolumn{4}{c}{$\beta_1=1,\beta_2=1000$}\\
		\hline
32    & 4.4526E+00 & --  & 1.5156E-02 &--  & 4.7280E+00 & --  & 1.6586E-02 &--  \\
64    & 2.2076E+00 & 1.0122  & 3.7762E-03 & 2.0049  & 2.5031E+00 & 0.9175  & 4.0744E-03 & 2.0253  \\
128   & 1.0987E+00 & 1.0066  & 9.4179E-04 & 2.0035  & 1.2222E+00 & 1.0343  & 9.9812E-04 & 2.0293  \\
256   & 5.4887E-01 & 1.0013  & 2.3539E-04 & 2.0004  & 6.0429E-01 & 1.0161  & 2.4728E-04 & 2.0131  \\
512   & 2.7433E-01 & 1.0005  & 5.8841E-05 & 2.0001  & 2.9700E-01 & 1.0248  & 6.1329E-05 & 2.0115  \\
		\hline
	\end{tabular}
	\label{table6}
\end{table}

\begin{table}[H]
	\centering
	\caption{Example \ref{li2}: Numerical errors and convergence rates on triangular meshes.}
	\centering
	\begin{tabular}{ccccc|cccc}
		\hline
		n&$\3bar u -u_h \3bar_h$&order&$\| u -u_0\|$&order&$\3bar u -u_h \3bar_h$&order&$\| u -u_h\|$&order\\
		\hline		
		\multicolumn{5}{c|}{$\beta_1=1,\beta_2=1$}&\multicolumn{4}{c}{$\beta_1=1,\beta_2=10$}\\
		\hline
32    & 5.1660E+00 &-- & 2.2684E-02 & -- & 4.8087E+00 &--  & 2.0643E-02 & --  \\
        64    & 2.5853E+00 & 0.9987  & 5.8433E-03 & 1.9568  & 2.4060E+00 & 0.9990  & 5.2895E-03 & 1.9644  \\
        128   & 1.2928E+00 & 0.9998  & 1.4853E-03 & 1.9760  & 1.2031E+00 & 0.9999  & 1.3397E-03 & 1.9812  \\
        256   & 6.4638E-01 & 1.0001  & 3.7463E-04 & 1.9872  & 6.0155E-01 & 1.0001  & 3.3739E-04 & 1.9894  \\
        512   & 3.2317E-01 & 1.0001  & 9.4089E-05 & 1.9934  & 3.0076E-01 & 1.0001  & 8.4680E-05 & 1.9943  \\
		\hline
		\multicolumn{5}{c|}{$\beta_1=1,\beta_2=100$}&\multicolumn{4}{c}{$\beta_1=1,\beta_2=1000$}\\
		\hline
32    & 4.8839E+00 & -- & 2.0936E-02 & --  & 9.2708E+00 & --& 3.1771E-02 & --  \\
		64    & 2.4184E+00 & 1.0140  & 5.3325E-03 & 1.9731  & 3.4381E+00 & 1.4311  & 6.9808E-03 & 2.1863  \\
		128   & 1.2045E+00 & 1.0056  & 1.3419E-03 & 1.9905  & 1.4184E+00 & 1.2773  & 1.4754E-03 & 2.2423  \\
		256   & 6.0168E-01 & 1.0014  & 3.3750E-04 & 1.9913  & 6.6666E-01 & 1.0893  & 3.5938E-04 & 2.0375  \\
		512   & 3.0069E-01 & 1.0007  & 8.4657E-05 & 1.9952  & 3.2281E-01 & 1.0463  & 8.8456E-05 & 2.0225  \\
		
		\hline
	\end{tabular}
	\label{table7}
\end{table}

In contrast to many existing examples in which the tangential flux jump vanishes,
$([\![\beta]\!]\nabla u \cdot \bm{\tau} = 0$ on $\Gamma)$,  
this example involves a nonzero one, i.e., 
$[\![\beta]\!]\nabla u \cdot \bm{\tau} \neq 0$ on $\Gamma$. In \cite{NCFEM_2023}, the authors pointed out that the numerical examples with $[\![\beta]\!]\nabla u \cdot \bm{\tau} \neq 0$ on $\Gamma$ may affect the performance of the proposed method, that is, optimal convergence may not always be guaranteed. The numerical results are shown in Tables \ref{table6} - \ref{table7} based on the triangular meshes (see Figure \ref{partition_2_li2} ). From these tables, we observe that the approximate capability of the IFE function and the convergence orders remain optimal in both the $H^1$ norm and $L^2$ norm. Figure \ref{Figure3_IEI} shows that, the condition numbers of the stiffness matrices exhibit an approximately $h^{-2}$ growth trend as the mesh is refined.

\begin{figure}[H]
	\centering
	\begin{minipage}[t]{0.48\linewidth}
		\centering
		\includegraphics[width=1.0\linewidth]{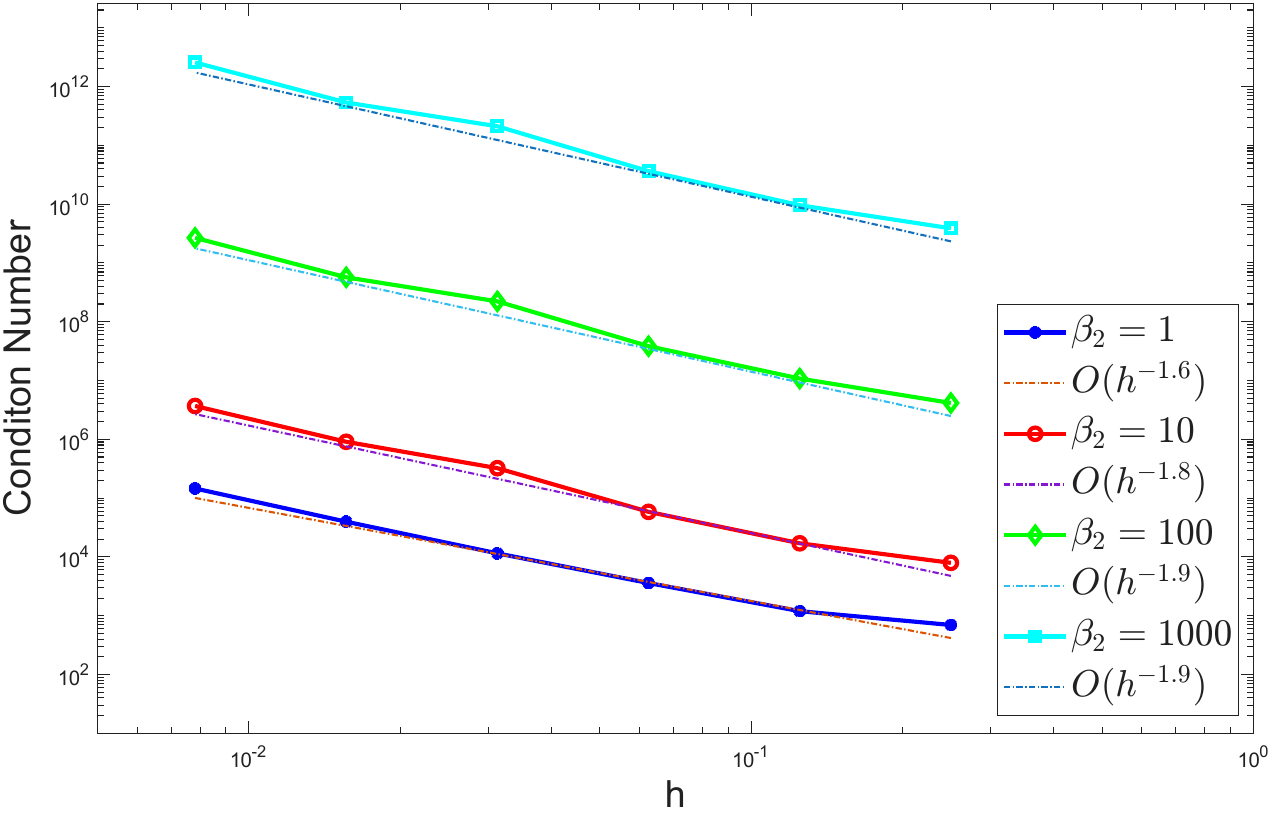}
	\end{minipage}
	\caption{Example \ref{li2}: The Condition numbers of the stiffness matrices with $\beta_1=1$.}
	\label{Figure3_IEI}
\end{figure}

\begin{example}\label{li4}{\rm({\textbf{Small-cut element}})}
 We consider an interface problem defined on the square domain $[-1,1]\times[-1,1]$. The interface function is defined as
 $$y-\delta=0,$$ 
 where 
 $$\delta = \dfrac{1}{40*2^{\ell}}, \quad \ell=0,1,2,\cdots,8.$$
 Let $\Omega_1=\left\{(x,y)\in \Omega| y > \delta \right\}$ and  $\Omega_2=\left\{(x,y)\in \Omega| y < \delta \right\}$. The exact solution is given by
 $$
 u = \left\{\begin{array}{lc}
 	\dfrac{1}{\beta_1}(y-\delta)\sin(\pi x)\sin(\pi y), & (x,y)\in \Omega_1, \\
 	\dfrac{1}{\beta_2}(y-\delta)\sin(\pi x)\sin(\pi y), & (x,y)\in \Omega_2.
 \end{array}
 \right.
 $$
  
\end{example}

\begin{figure}[H]
	\centering
	\begin{minipage}[t]{0.3\linewidth}
		\centering
		\includegraphics[width=1.0\linewidth]{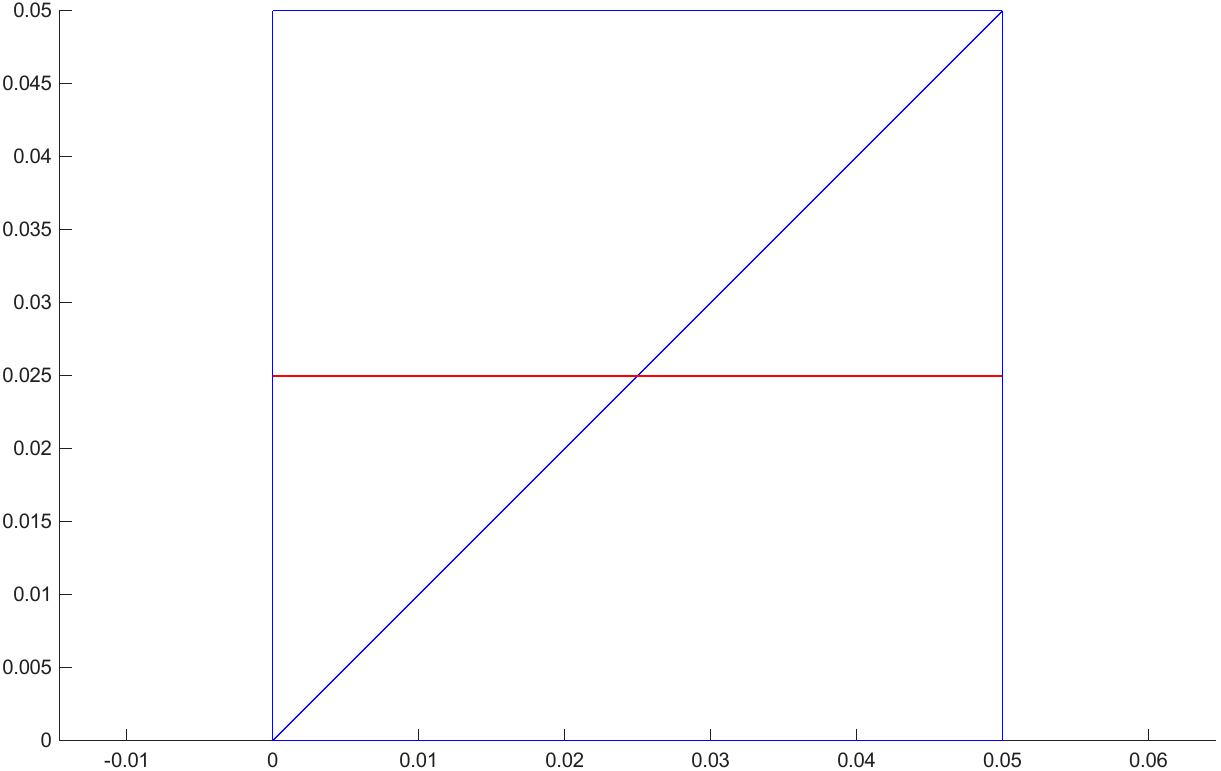}
	\end{minipage}
	\begin{minipage}[t]{0.3\linewidth}
		\centering
		\includegraphics[width=1.0\linewidth]{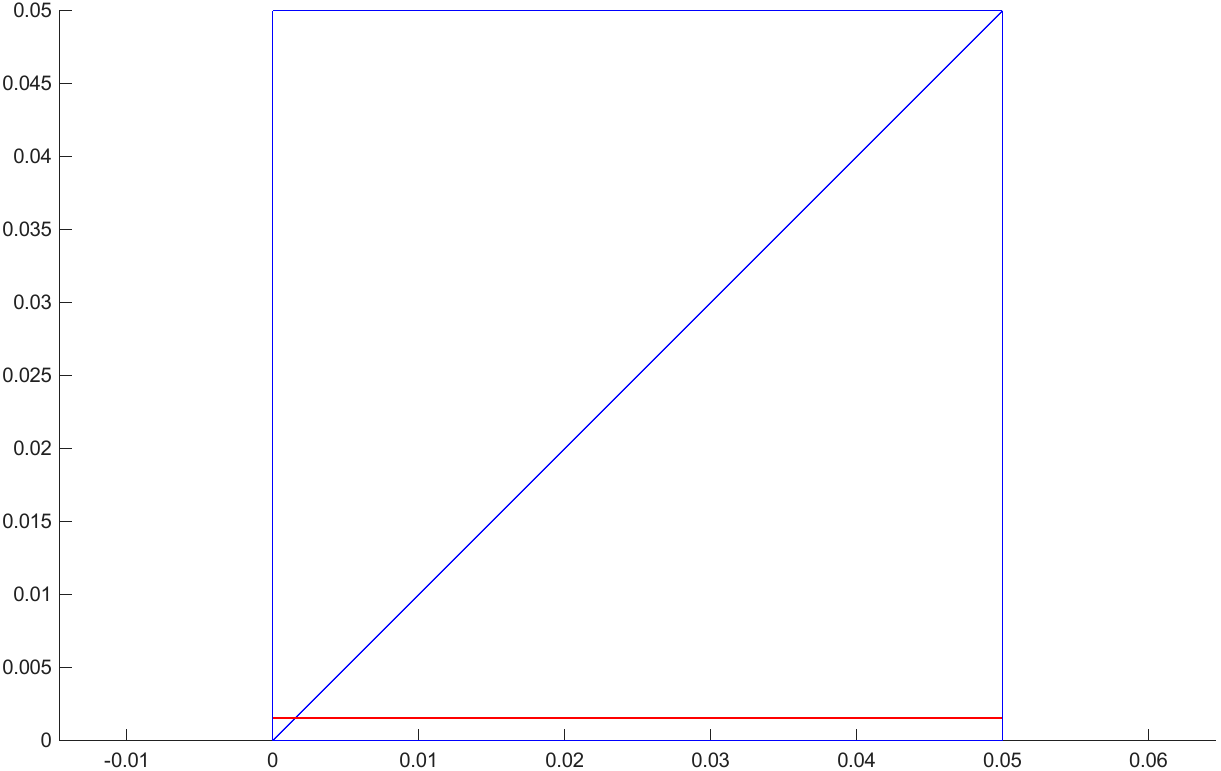}
	\end{minipage}
	\begin{minipage}[t]{0.3\linewidth}
		\centering
		\includegraphics[width=1.0\linewidth]{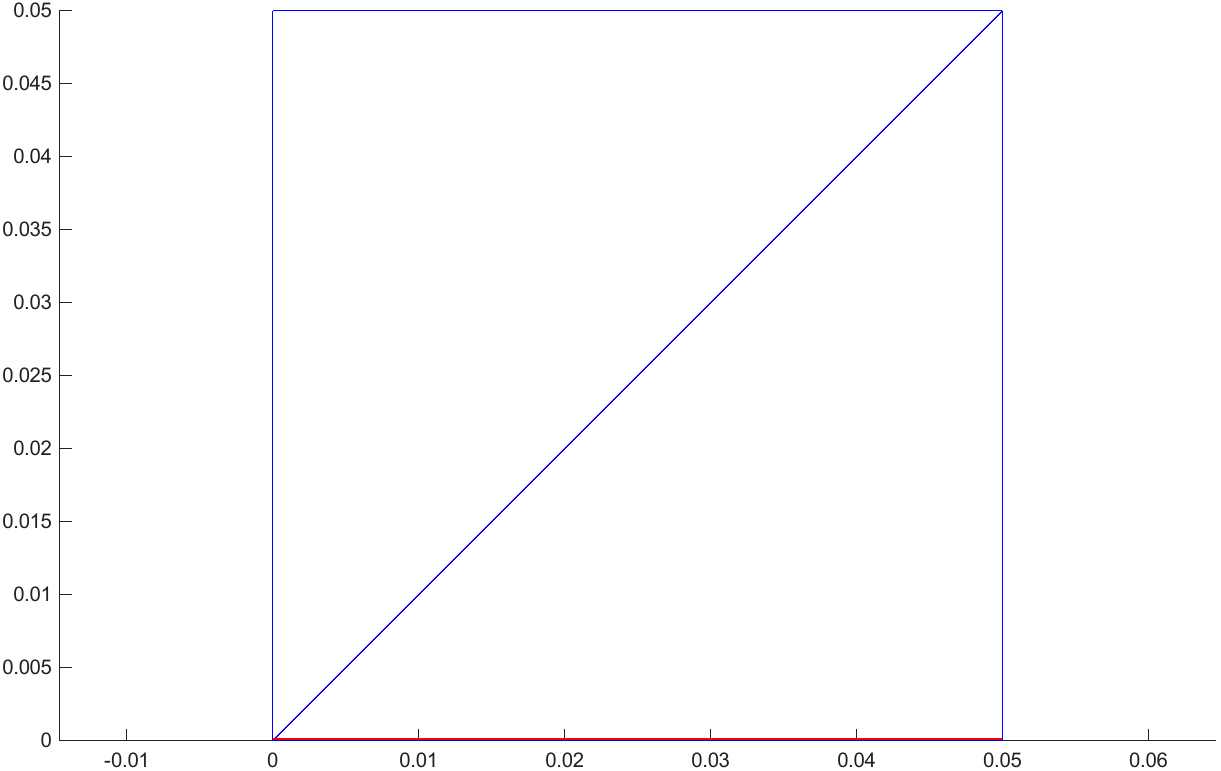}
	\end{minipage}
	\caption{Example \ref{li4}: The local partition  with $\delta= 1/(40*2^{\ell})$. (Left: $\ell=0$, Middle: $\ell=4$, Right:  $\ell=8$).}
	\label{Figure2_small_cut_partition}
\end{figure}

\begin{figure}[H]
	\centering
	\begin{minipage}[t]{0.3\linewidth}
		\centering
		\includegraphics[width=1.0\linewidth]{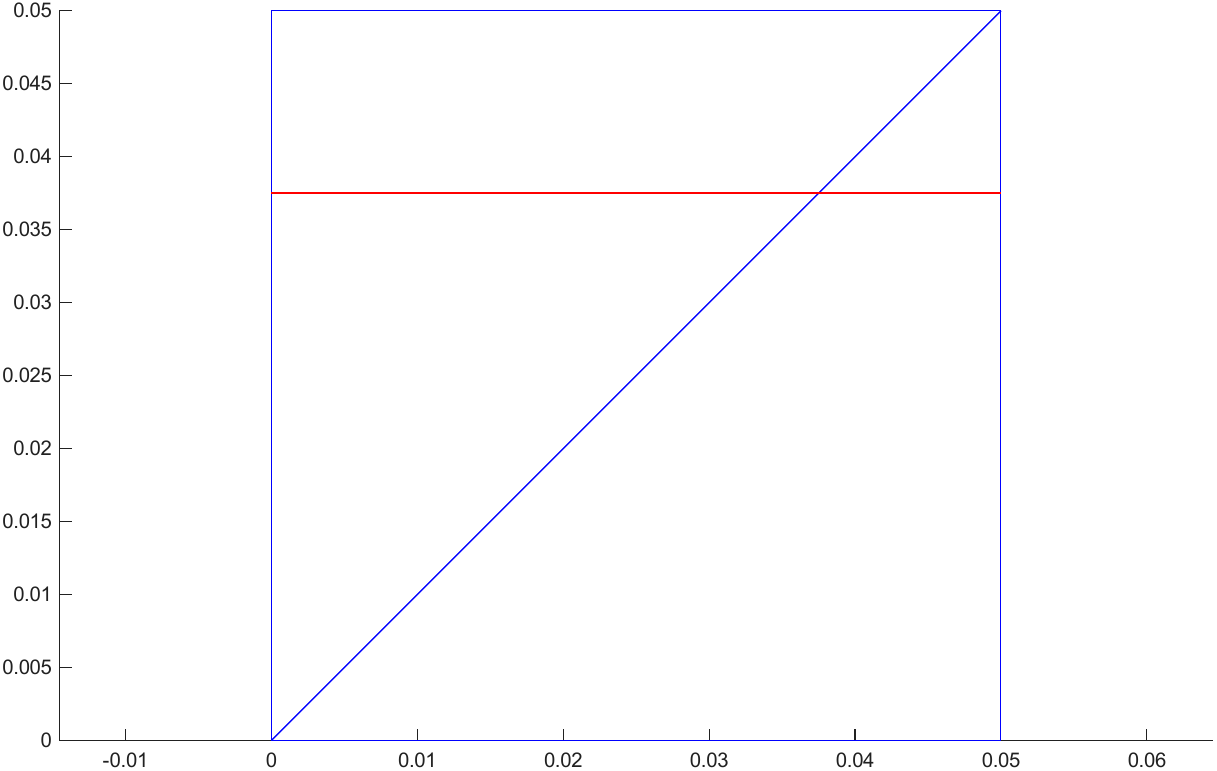}
	\end{minipage}
	\begin{minipage}[t]{0.3\linewidth}
		\centering
		\includegraphics[width=1.0\linewidth]{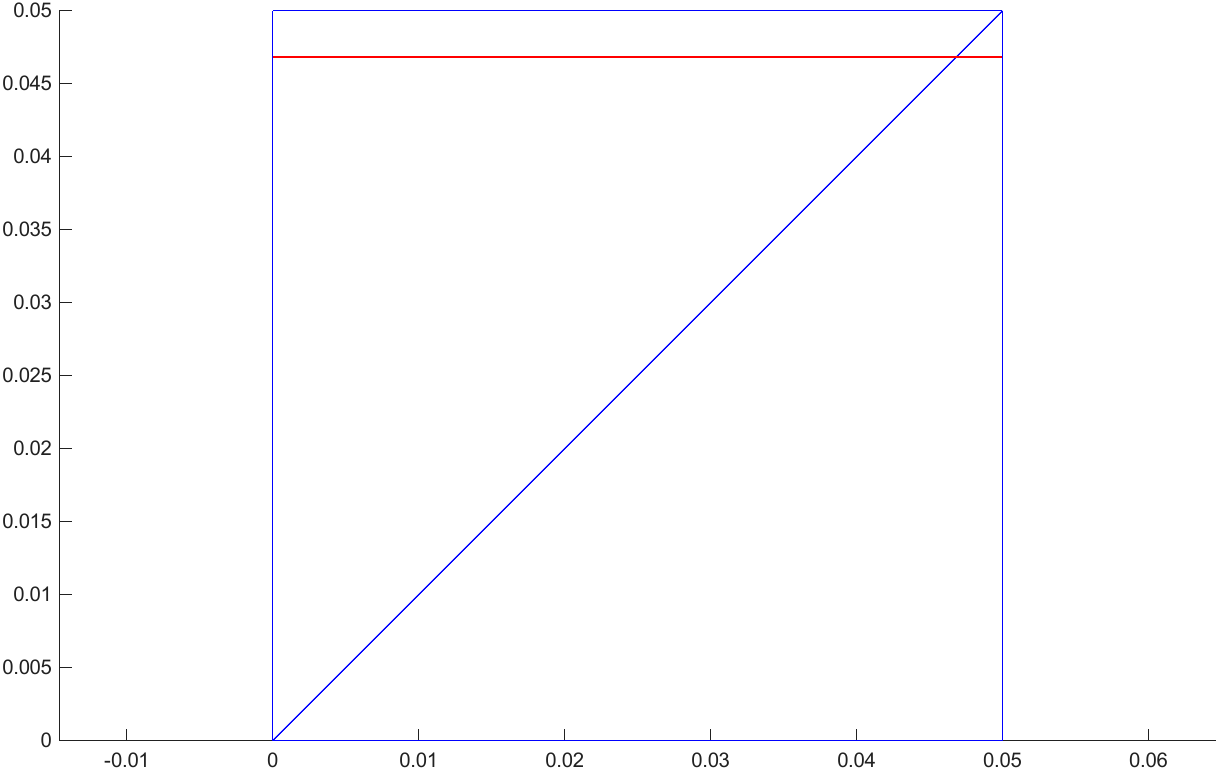}
	\end{minipage}
	\begin{minipage}[t]{0.3\linewidth}
		\centering
		\includegraphics[width=1.0\linewidth]{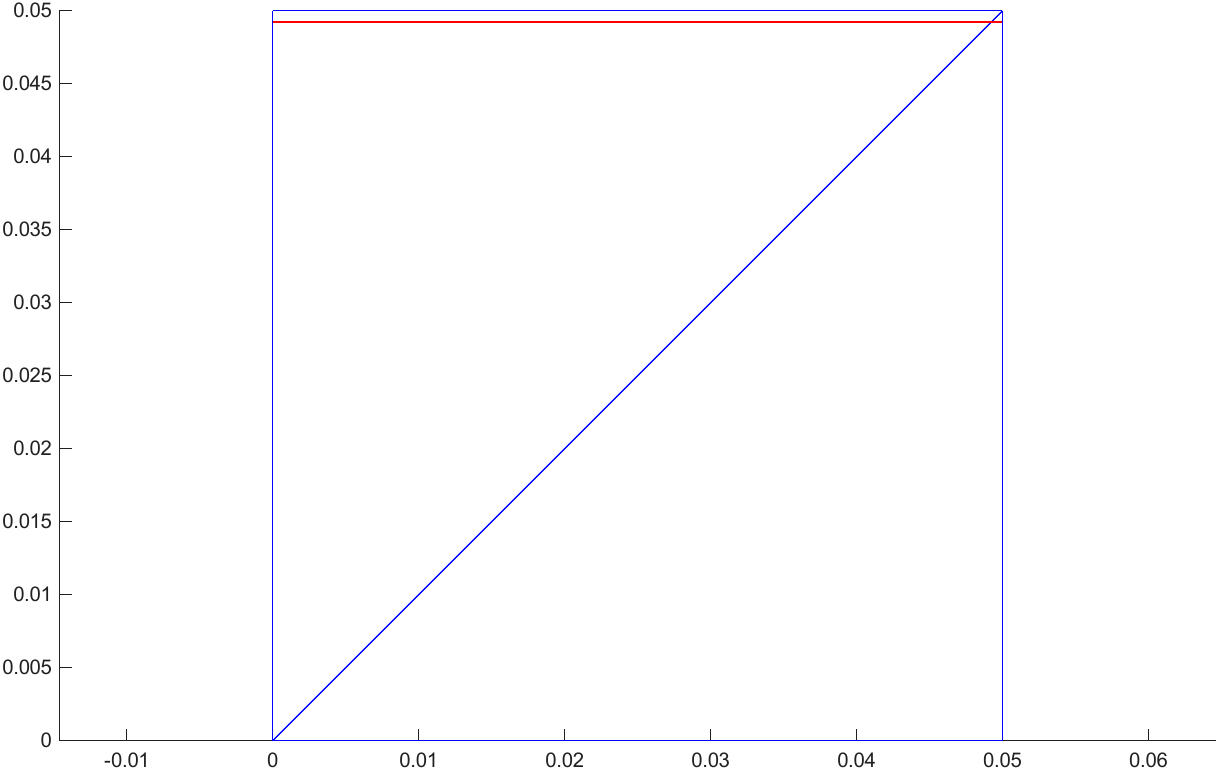}
	\end{minipage}
	\caption{Example \ref{li4}: The local partition  with $\delta=1/20 - 1/(40*2^{\ell})$. (Left: $\ell=1$, Middle: $\ell=3$, Right:  $\ell=5$).}
	\label{Figure3_small_cut_partition}
\end{figure}

\begin{figure}[H]
	\centering
	\begin{minipage}[t]{0.48\linewidth}
		\centering
		\includegraphics[width=1.0\linewidth]{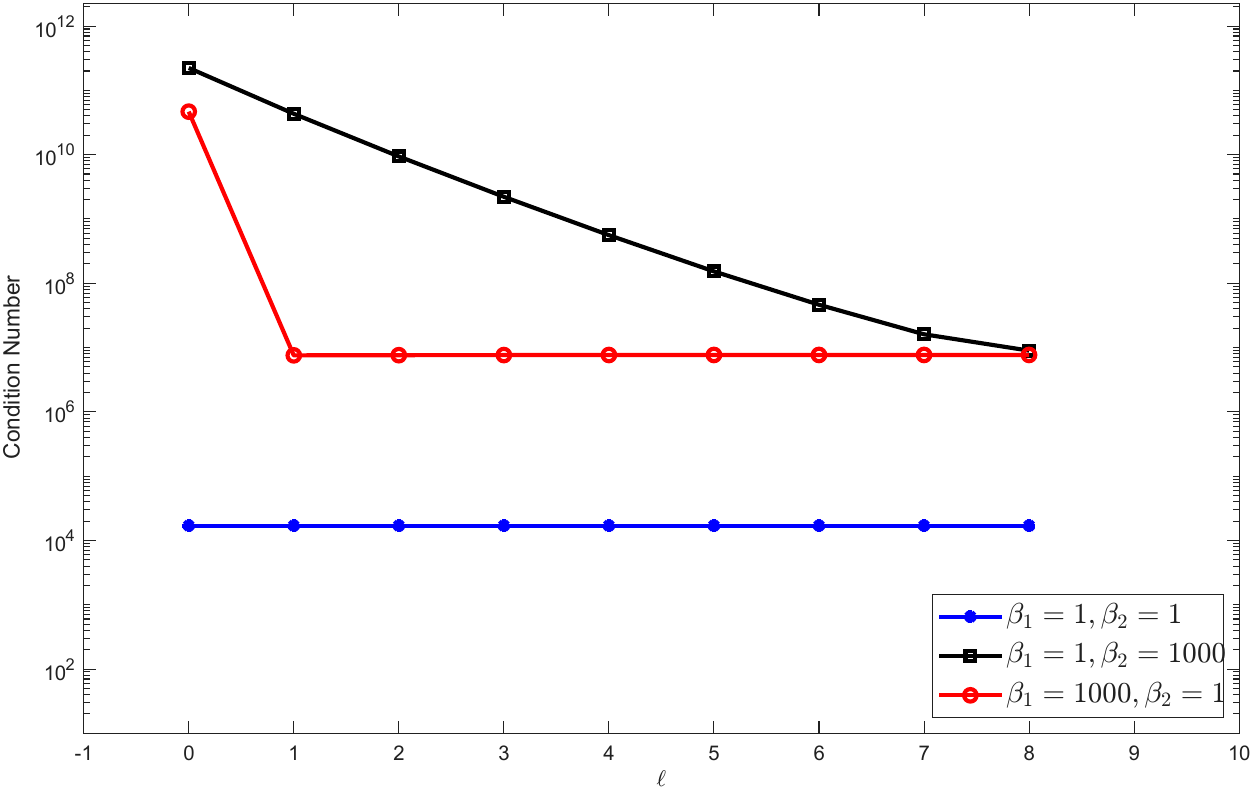}
	\end{minipage}
	\begin{minipage}[t]{0.48\linewidth}
		\centering
		\includegraphics[width=1.0\linewidth]{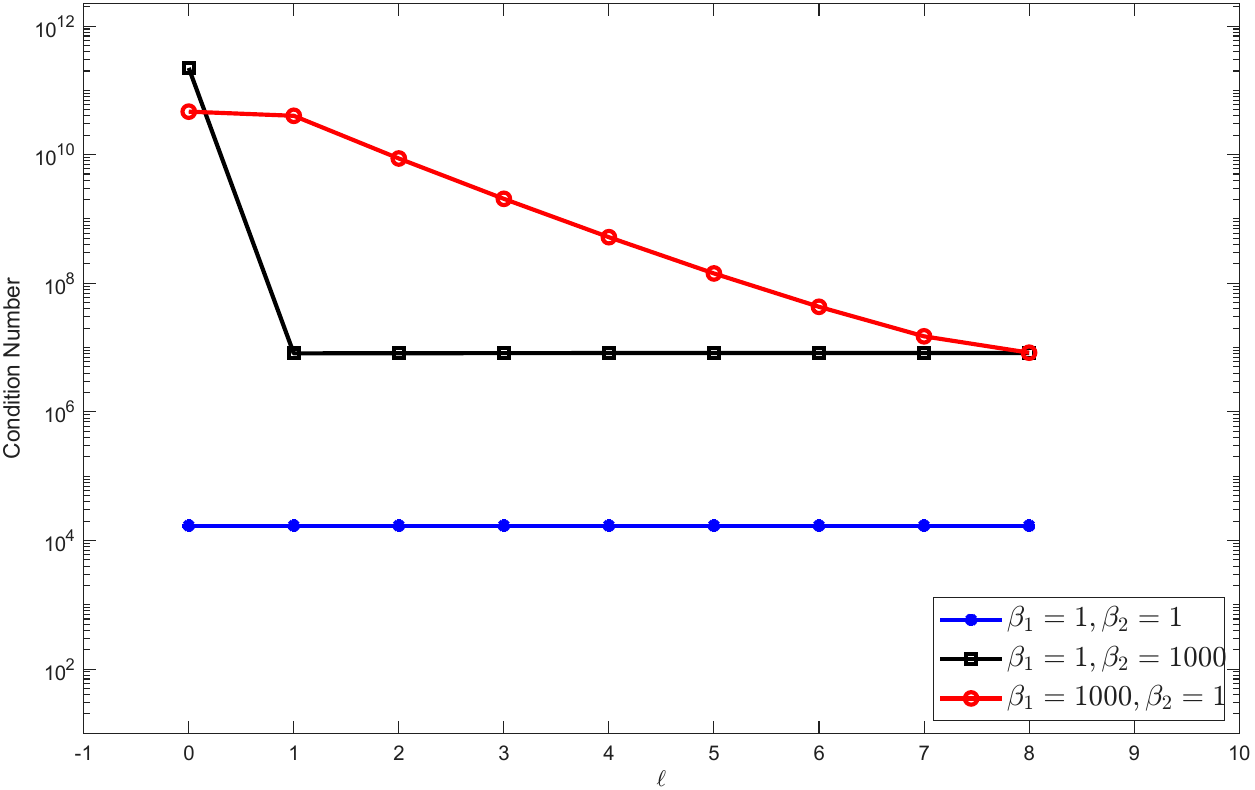}
	\end{minipage}
	\caption{Example \ref{li4}: The Condition numbers with $n=40$. (Left: $\delta= 1/(40*2^{\ell})$, Right: $\delta=1/20 - 1/(40*2^{\ell})$).}
	\label{Figure2_small_cut}
\end{figure}

In this example, we further investigate whether the condition numbers of the stiffness matrices depend on the position of the interface within an element, namely, whether they are affected by small-cut elements. As the parameter $\ell$ increases, the portion of subdomain $\Omega_2$ inside the interface element gradually decreases (see Figure \ref{Figure2_small_cut_partition}). In this case, the IFE function \eqref{IFE_function2} is employed. Taking $n=40$, we examine the condition numbers under three different coefficients: $\beta_1=\beta_2$, $\beta_1>\beta_2$, and $\beta_1<\beta_2$. It can be seen from Figure \ref{Figure2_small_cut} that the condition numbers are unaffected by small-cut elements when $\beta_1=\beta_2$ and $\beta_1>\beta_2$. In contrast, the condition numbers decrease as $\ell$ increases  when $\beta_1<\beta_2$. 

We also investigate the condition numbers for three different coefficients in the case 
$$\delta=\dfrac{1}{20} - \dfrac{1}{40*2^{\ell}}, \quad \ell=0,1,2,\cdots,8.$$ 
Under this interface configuration, as the parameter $\ell$ increases, the portion of subdomain $\Omega_1$ inside the interface element gradually decreases (see Figure \ref{Figure3_small_cut_partition}), and the IFE function \eqref{IFE_function1} is used. The numerical results in Figure \ref{Figure2_small_cut} show that, when $\beta_1=\beta_2$ and $\beta_1 < \beta_2$, the condition numbers are uninfluenced by the small-cut elements. However, when $\beta_1 > \beta_2$, the condition numbers decrease as $\ell$ increases. This observation is consistent with the results obtained for the previous interface $\delta= 1/(40*2^{\ell})$. This phenomenon is due to the fact that the integral values of the constructed IFE functions over the interface elements decrease, which exhibits a trend consistent with that of the condition numbers. Although the condition numbers vary, the results in Table \ref{table12} demonstrate that the convergence is not affected, and the convergence rates are still optimal.

\begin{table}[H]
	\centering
	\caption{Example \ref{li4}: Projection errors and numerical errors on triangular meshes with $\beta_1=1$ and $\beta_2=1000$.}
	\centering
	\begin{tabular}{ccccccccc}
		\hline
		n&$\|\Pi_h u -u \|_{1}$&order&$\|\Pi_h u -u\|$&order&$\3bar u -u_h \3bar_h$&order&$\| u -u_h\|$&order\\
		\hline		
		\multicolumn{9}{c}{$\ell=1$}\\
		\hline
            10    & 6.5025E-01 & --  & 4.9276E-02 & --  & 4.9492E-01 & --  & 1.2234E-02 & --  \\
        20    & 3.3329E-01 & 0.9642  & 1.3033E-02 & 1.9187  & 2.5065E-01 & 0.9815  & 3.0997E-03 & 1.9807  \\
        30    & 2.2324E-01 & 0.9883  & 5.8554E-03 & 1.9733  & 1.6751E-01 & 0.9940  & 1.3811E-03 & 1.9938  \\
        40    & 1.6772E-01 & 0.9941  & 3.3064E-03 & 1.9867  & 1.2575E-01 & 0.9968  & 7.7754E-04 & 1.9969  \\
        50    & 1.3429E-01 & 0.9963  & 2.1199E-03 & 1.9919  & 1.0065E-01 & 0.9977  & 4.9785E-04 & 1.9980  \\
        60    & 1.1196E-01 & 0.9973  & 1.4736E-03 & 1.9945  & 8.3905E-02 & 0.9981  & 3.4582E-04 & 1.9985  \\
        70    & 9.5999E-02 & 0.9978  & 1.0833E-03 & 1.9959  & 7.1939E-02 & 0.9982  & 2.5413E-04 & 1.9985  \\
        
		\hline
		\multicolumn{9}{c}{$\ell=4$}\\
		\hline
	    10    & 6.4232E-01 & --  & 4.8648E-02 & --  & 4.8922E-01 & --  & 1.2090E-02 & -- \\
	    20    & 3.2926E-01 & 0.9641  & 1.2870E-02 & 1.9184  & 2.4777E-01 & 0.9815  & 3.0634E-03 & 1.9806  \\
	    30    & 2.2054E-01 & 0.9884  & 5.7822E-03 & 1.9733  & 1.6557E-01 & 0.9941  & 1.3649E-03 & 1.9938  \\
	    40    & 1.6568E-01 & 0.9942  & 3.2649E-03 & 1.9867  & 1.2428E-01 & 0.9971  & 7.6844E-04 & 1.9970  \\
	    50    & 1.3265E-01 & 0.9965  & 2.0933E-03 & 1.9920  & 9.9466E-02 & 0.9983  & 4.9200E-04 & 1.9982  \\
	    60    & 1.1059E-01 & 0.9977  & 1.4551E-03 & 1.9947  & 8.2906E-02 & 0.9988  & 3.4174E-04 & 1.9988  \\
    	70    & 9.4812E-02 & 0.9984  & 1.0697E-03 & 1.9962  & 7.1071E-02 & 0.9992  & 2.5111E-04 & 1.9991  \\
		\hline
		\multicolumn{9}{c}{$\ell=8$}\\
		\hline
	    10    & 6.4127E-01 & --  & 4.8564E-02 &--  & 4.8847E-01 & --  & 1.2071E-02 & -- \\
     	20    & 3.2872E-01 & 0.9641  & 1.2848E-02 & 1.9183  & 2.4740E-01 & 0.9814  & 3.0586E-03 & 1.9806  \\
	    30    & 2.2019E-01 & 0.9884  & 5.7725E-03 & 1.9733  & 1.6533E-01 & 0.9941  & 1.3628E-03 & 1.9938  \\
    	40    & 1.6541E-01 & 0.9942  & 3.2594E-03 & 1.9867  & 1.2410E-01 & 0.9971  & 7.6723E-04 & 1.9969  \\
	    50    & 1.3243E-01 & 0.9965  & 2.0898E-03 & 1.9920  & 9.9316E-02 & 0.9983  & 4.9123E-04 & 1.9982  \\
	    60    & 1.1041E-01 & 0.9977  & 1.4526E-03 & 1.9947  & 8.2780E-02 & 0.9989  & 3.4120E-04 & 1.9988  \\
	    70    & 9.4659E-02 & 0.9984  & 1.0679E-03 & 1.9962  & 7.0963E-02 & 0.9992  & 2.5071E-04 & 1.9991  \\
		\hline
	\end{tabular}
	\label{table12}
\end{table}

\section{Conclusion}
In this paper, the linear explicit IFE functions are constructed to exactly satisfy the interface conditions on the actual interface. The proposed IFE functions avoids solving local linear systems on interface elements and therefore does not introduce additional computational cost. We prove that the constructed IFE functions achieve optimal approximation properties. Based on these IFE functions, we develop an immersed SIPDG method for second-order elliptic interface problems with homogeneous interface conditions and prove that optimal error estimates are derived in both the $H^1$ norm and $L^2$ norm. Numerical experiments are presented to confirm the theoretical analysis and demonstrate the effectiveness of the numerical scheme. In particular, the numerical results show that the condition numbers of the stiffness matrices are insensitive to small-cut elements. In future work, we plan to extend the linear explicit IFE functions proposed in this paper to three-dimensional interface problems.

\bibliographystyle{siam}
\bibliography{lib}

\newpage
\end{document}
\endinput